\documentclass[11pt]{amsart}
\usepackage{amsmath} 
\usepackage{amsthm}
\usepackage{amssymb} 
\usepackage{amsfonts}
\usepackage{graphicx}
\usepackage{color}
\usepackage[normalem]{ulem}

\newtheorem{theorem}{Theorem}[section]
\newtheorem{lemma}[theorem]{Lemma}
\newtheorem{proposition}[theorem]{Proposition}

\newtheorem{question}[theorem]{Question}

\newtheorem{corollary}[theorem]{Corollary}
\newtheorem{claim}[theorem]{Claim}
\theoremstyle{definition}
\newtheorem{definition}[theorem]{Definition}

\newcommand{\Z}{\mathbb{Z}}
\newcommand{\R}{\mathbb{R}}
\newcommand{\RP}{\R P}
\let\int\relax
\newcommand{\int}{\mathring}

\newcommand{\mirror}[1]{\overline{#1}}
\newcommand{\boundary}{\partial}

\usepackage{accents}

\DeclareMathSymbol{\wtilde}{\mathord}{largesymbols}{"65}
\newcommand\lowerwtilde{%
  \text{\smash{\raisebox{-1.3ex}{%
    $\wtilde$}}}}
\DeclareRobustCommand\newtilde[1]{%
  \mathchoice
    {\accentset{\displaystyle\lowerwtilde}{#1}}
    {\accentset{\textstyle\lowerwtilde}{#1}}
    {\accentset{\scriptstyle\lowerwtilde}{#1}}
    {\accentset{\scriptscriptstyle\lowerwtilde}{#1}}
}
\title[Trisections and the Price twist]{Trisections of surface complements and the Price twist}
\author[Seungwon Kim and Maggie Miller]{Seungwon Kim and Maggie Miller}

\address{Seungwon Kim\\National Institute for Mathematical Sciences\\Daejeon, South Korea}
\email{math751@gmail.com}
\address{Maggie Miller\\Department of Mathematics\\Princeton University\\ Princeton, NJ 08544, USA}
\email{maggiem@math.princeton.edu}
\urladdr{http://www.math.princeton.edu/~maggiem}

\begin{document}
\maketitle

\begin{abstract}
Given a real projective plane $S$ embedded in a $4$-manifold $X^4$ with Euler number $2$ or $-2$, the Price twist is a surgery operation on $\nu(S)$ yielding (up to) three different $4$-manifolds: $X^4,\tau_S(X^4),\Sigma_S(X^4)$. This is of particular interest when $X^4=S^4$, as then $\Sigma_S(X^4)$ is a homotopy 4-sphere which is not obviously diffeomorphic to $S^4$. In this paper, we show how to produce a trisection description of each Price twist on $S\subset X^4$ by producing a relative trisection of $X^4\setminus\nu(S)$. Moreover, we show how to produce a trisection description of general surface complements in $4$-manifolds.

\end{abstract}
\setcounter{tocdepth}{2}
\setcounter{equation}{0}
\section{Introduction}

In 2012, Gay and Kirby \cite{gaykirby} introduced  \emph{trisections} of closed $4$-manifolds, an analogue of Heegaard splittings of $3$-manifolds. During the past six years, topologists have extended this idea to various objects such as $4$-manifolds with boundary \cite{nickthesis}, knotted surfaces in $4$-manifolds \cite{newsurfacetrisections}, and finitely presented groups \cite{groups}. Furthermore, trisections have been used to study classical problems in topology, such as the generalized property R conjecture \cite{propr} and the Thom conjecture \cite{peter}.  Recently, Gay and Meier \cite{jeffgluck} have studied surgery on spheres in $S^4$ (including the Gluck twist and blowdown), 
by constructing trisections of sphere complements in $4$-manifolds. In this paper, we study the \emph{Price twist}, which is a generalization of the Gluck twist \cite{pricepaper} (surgering an $\RP^2$ rather than an $S^2$) by constructing trisections of general surface complements in $4$-manifolds.

This paper is organized as follows. In Section \ref{sec:trisection}, we give definitions of various notions of trisection. In Section \ref{sec:price}, we describe the Price twist. In Section \ref{sec:complement}, we show how to produce a relative trisection of a surface complement in a $4$-manifold. Lastly, in Section \ref{sec:gettrisection} we give a procedure that produces a trisection of a $4$-manifold arising from a Price twist.

\subsection*{Acknowledgements}
Thanks to Jeff Meier, David Gay, and Alex Zupan for helpful conversations about surface complements (especially at CIRM 2018 and BIRS-CMO 2017). Thanks to Selman Akbulut for making us aware of plugs in $4$-manifolds at the AMS 2018 Spring Eastern Sectional Meeting. The second author also thanks her graduate advisor, David Gabai. Finally, we thank an anonymous referee for thoroughly reading this paper and providing many helpful comments.

The first author is supported by the National Institute for Mathematical Sciences South Korea (NIMS). The second author is a fellow in the National Science Foundation Graduate Research Fellowship program, under Grant No. DGE-1656466.

\section{Trisection, relative trisection, and bridge trisection}\label{sec:trisection}
\subsection{Definitions}

First, we define a trisection of a closed $4$-manifold.
\begin{definition}\cite{gaykirby}
Let $X^4$ be a closed $4$-manifold. A {\emph{$(g,k)$-trisection}} of $X^4$ is a triple $(X_1,X_2,X_3)$ where
\begin{itemize}
\item $X_1\cup X_2\cup X_3=X^4$,
\item $X_i\cong\natural_k S^1\times B^3$,
\item $X_i\cap X_j=\boundary X_i\cap\boundary X_j\cong \natural_g S^1\times B^2$
\item $X_1\cap X_2\cap X_3\cong \Sigma_g$,
\end{itemize}

where $\Sigma_g$ is the closed orientable surface of genus $g$.
\end{definition}

Note that from the definition, $(\Sigma_g, X_i\cap X_j, X_i\cap X_k)$ gives a Heegaard splitting of $\partial X_i\cong\#_k (S^1\times S^2)$. By Laudenbach-Poenaru \cite{laudenbach}, $X^4$ is specified by its {\emph{spine}}, $(\Sigma_g\times D^2)\cup_{i,j}(X_i\cap X_j)$. Therefore, we usually describe a trisection $(X_1,X_2,X_3)$ by a {\emph{trisection diagram}} $(\Sigma_g,\alpha,\beta,\gamma)$ where each of $\alpha,\beta,$ and $\gamma$ consist of $g$ independent curves bounding disks in the handlebodies $X_1\cap X_2, X_2\cap X_3, X_1\cap X_3$ respectively. We generally depict this diagram by drawing on $\Sigma_g$ the $\alpha$ curves in red, the $\beta$ curves in blue, and the $\gamma$ curves in green. One typically uses the names $H_\alpha:=X_1\cap X_2, H_\beta:=X_2\cap X_3, H_\gamma:=X_3\cap X_1$ for the double intersections. In words, $H_\alpha$ is usually called ``the $\alpha$ handlebody'' (and similarly $H_\beta$ is the $\beta$ handlebody and $H_\gamma$ is the $\gamma$ handlebody).

For expositions of trisections, refer to \cite{gaykirby} or \cite{propr}.

Next, we define a relative trisection of a compact $4$-manifold with boundary.
\begin{definition}\label{relativedefinition}\cite{gaykirby}
Let $X^4$ be a compact $4$-manifold with boundary $M^3\neq\emptyset$. A \emph{$(g,k,p,b)$-relative trisection} of $X^4$ is a triple $(X_1,X_2,X_3)$ where

\begin{itemize}
\item $X_1\cup X_2\cup X_3=X^4$,
\item $X_i\cong\natural_k S^1\times B^3$,
\item $X_i\cap X_j=\boundary X_i\cap\boundary X_j\cong (\Sigma_g^b\times I)\cup((g-p)$ $2$-handles)\\$\cong\natural_{g+b+p-1} S^1\times B^2$
\item $X_1\cap X_2\cap X_3\cong \Sigma_g^b$,
\end{itemize}

where $\Sigma_g^b$ is an orientable surface of genus $g$ and $b>0$ boundary components. In $X_i\cap X_j$, the $(g-p)$ $2$-handles are attached to cancel $(g-p)$ $1$-handles of $\Sigma_g^b \times I$. This ensures $X_i\cap X_j$ is a $3$-dimensional handlebody of genus $(g+p+b-1)$. Moreover, we have the following conditions on $M^3$:

\begin{itemize}
\item $X_i\cap X_j\cap M^3\cong\Sigma_p^b$,
\item $X_i\cap M^3=(X_i\cap X_j\cap M^3)\times I$.
\end{itemize}

Thus, $(X_1,X_2,X_3)$ determines an open book decomposition on $M^3$, where each $X_i\cap X_j\cap M^3$ is a single page. (In particular, if $M$ is connected, then $X_i\cap X_j\cap M^3$ must be connected.)

Relative trisections of $X^4$ and $Y^4$ with $\boundary X^4 \cong \boundary Y^4 $ can be glued to form a trisection of $X^4 \cup_\boundary Y^4 $ if and only if the relative trisections induce the same (isotopic) open books on $\boundary X^4 \cong\boundary Y^4 $, but with opposite orientations \cite{nickthesis}.

\end{definition}

Again, $X^4$ is specified by its {\emph{spine}}, $(\Sigma_g^b\times D^2)\cup_{i,j}(X_i\cap X_j)$ \cite{monodromypaper}. Therefore, we usually describe a relative trisection $(X_1,X_2,X_3)$ by a {\emph{relative trisection diagram}} $(\Sigma_g^b,\alpha,\beta,\gamma)$ where each of $\alpha,\beta,\gamma$ consist of $(g-p)$ independent curves bounding disks in the compression bodies $X_1\cap X_2, X_2\cap X_3, X_1\cap X_3$ respectively. We generally depict this diagram by drawing on $\Sigma_g^b$ the $\alpha$ curves in red, the $\beta$ curves in blue, and the $\gamma$ curves in green. One typically uses the names $H_\alpha:=X_1\cap X_2, H_\beta:=X_2\cap X_3, H_\gamma:=X_3\cap X_1$ for the double intersections. In words, $H_\alpha$ is usually called ``the $\alpha$ compression body'' (and similarly $H_\beta$ is the $\beta$ compression body and $H_\gamma$ is the $\gamma$ compression body).

For expositions of relative trisections, refer to \cite{nickthesis}, \cite{monodromypaper} or \cite{lefschetz}.

Lastly, we end this section by defining a bridge trisection of a knotted surface in an arbitrary $4$-manifold. 

\begin{definition}
Let $S$ be a surface embedded in a $4$-manifold $X^4$. Say $X^4$ has a $(g,k)$-trisection $(X_1,X_2,X_3)$. By \cite{newsurfacetrisections}, $S$ can be isotoped so that
\begin{itemize}
\item $S\cap X_i$ is a disjoint union of $c$-boundary parallel disks,
\item $S\cap X_i\cap X_j$ is a trivial tangle of $b$ arcs.
\end{itemize}

Note $\chi(S)=3c-b$. We say $S$ is in{\emph{$(c,b)$-bridge position}} in $X^4$ with respect to $(X_1,X_2,X_3)$, or that $(X_1,X_2,X_3)$ induces a \emph{$(c,b)$-bridge trisection} on $S$. We can stabilize $(X_1,X_2,X_3)$ to find a trisection $(\newtilde{X}_1,\newtilde{X}_2,\newtilde{X}_3)$ which induces a $(1,b-3(c-1))$-bridge trisection on $S$ \cite{newsurfacetrisections}.
\end{definition}

For expositions on bridge trisections of surfaces in $4$-manifolds, see \cite{surfacetrisections} or \cite{newsurfacetrisections}.

Bridge trisections admit two kinds of diagrams. One is a diagram (in the standard sense) of the triple of tangles $(X_i\cap X_{i+1},S\cap X_i\cap X_{i+1})\mid_{i=1,2,3}$ (where $X_4:=X_1$). When $g=0$, this diagram consists of $3$-tangle diagrams in disks. This is usally called a {\emph{triplane diagram}}. Another kind of diagram is the {\emph{shadow diagram}}. The shadow diagram is less likely to be familiar to the average reader, so we expand more upon its definition.

\begin{definition}\label{shadowdef}
Let $S$ be a surface in $(c,b)$-bridge position in $X^4$ with respect to the $(g,k)$-trisection $(X_1,X_2,X_3)$. Let $(\Sigma_g,\alpha,\beta,\gamma)$ be a trisection diagram of $(X_1,X_2,X_3)$. Identify $\Sigma_g$ with $X_1\cap X_2\cap X_3$. A {\emph{shadow diagram}} for $S$ is a septuple $(\Sigma_g,\alpha,\beta,\gamma,s_\alpha,s_\beta,s_\gamma)$ where
\begin{itemize}
\item Each $s_*$ is a collection of $b$ disjoint arcs in $\Sigma_g$ with $\boundary s_*=S\cap\Sigma_g$.
\item The collection of $c$ circles $(S\cap(X_1\cap X_2))\cup s_\alpha$ bounds a set of $c$ disjoint embedded disks $D_\alpha$ in $X_1\cap X_2$, with $\int D_\alpha\cap(\boundary(X_1\cap X_2))=\emptyset$. That is, $s_\alpha$ can be obtained by projecting the boundary-parallel tangle $S\cap(X_1\cap X_2)$ onto $\boundary (X_1\cap X_2)=\Sigma$.
\item Similarly, $s_\beta$ and $s_\gamma$ can be obtained by projecting the boundary-parallel tangles $S\cap(X_2\cap X_3)\subset X_2\cap X_3$ and $S\cap(X_1\cap X_3)\subset X_1\cap X_3$ onto $\Sigma_g$, respectively.
\end{itemize}

In this paper, whenever we draw a shadow diagram $(\Sigma_g,\alpha,\beta,\gamma,s_\alpha,s_\beta,s_\gamma)$, we will draw the curves $\alpha,\beta,\gamma$ and arcs $s_\alpha, s_\beta,s_\gamma$ on the surface $\Sigma$. The $\alpha$ curves and $s_\alpha$ arcs will be red; the $\beta$ curves and $s_\beta$ arcs will be blue; the $\gamma$ curves and $s_\gamma$ arcs will be green. The endpoints of $s_*$ will be indicated with black dots. See Figure \ref{fig:shadowexample} for a small example of a shadow diagram.
\end{definition}

\begin{figure}
\includegraphics[width=.6\textwidth]{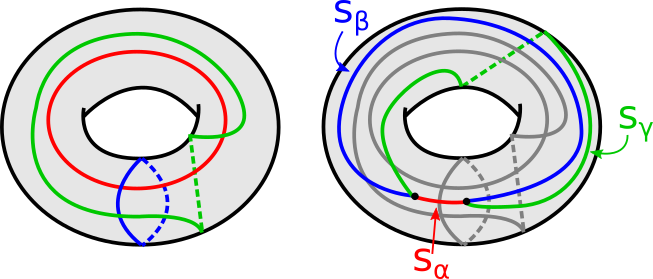}
\caption{Left: A trisection diagram $(\Sigma_1,\alpha,\beta,\gamma)$ for $\mathbb{C}P^2$. Right: A shadow diagram $(\Sigma_1,\alpha,\beta,\gamma,s_\alpha,s_\beta,s_\gamma)$ for a surface in $\mathbb{C}P^2$. The $\alpha, \beta,$ and $\gamma$ curves are visible in the left diagram, so we dim them here to make the shadow arcs $s_*$ more visible. The described surface is in $(1,1)$-bridge position, so the surface is a sphere. (In fact, this is a shadow diagram for the standard $\mathbb{C}P^1$.)}\label{fig:shadowexample}
\end{figure}

The shadow diagram $(\Sigma_g,\alpha,\beta,\gamma, s_\alpha,s_\beta,s_\gamma)$ determines each $X_i\cap X_j$ (up to diffeomorphism) and $S\cap(X_i\cap X_j)$ (up to isotopy). Since $S\cap X_i$ consists of boundary-parallel disks, this determines $S$ up to isotopy.

See \cite{surfacetrisections} and \cite{newsurfacetrisections} for the original definitions of triplane and shadow diagrams and many examples.

\subsection{Kirby diagrams from relative trisections}
\label{sec:getKirby}


In this subsection, we discuss how to obtain a Kirby diagram of $X^4$ from a relative trisection of $X^4$. This essentially comes from Lemma 14 of \cite{gaykirby} (although they only consider closed manifolds, the procedure is almost exactly the same for manifolds with boundary). An alternate viewpoint is found in \cite{propr}.

Let $(\Sigma,\alpha,\beta,\gamma)$ be a $(g,k,p,b)$-relative trisection of $X^4$. Consider the following handle structure on $X^4$:
\begin{itemize}
\item One $0$-handle and $k$ $1$-handles, glued to make $X_1$.
\item $(g-k+p+b-1)$ $2$-handles, corresponding to $\gamma$ curves which are dual to the $\beta$ curves (up to handle slides). These handles are glued to the $0$- and $1$-handles so that $X_2=(H_\beta\times I)\cup (2$-handles).
\item $(k-2p-b+1)$ $3$-handles, corresponding to parallel (up to handle slides) $\alpha,\gamma$ curves. Then $X_3=((H_\alpha\cup_\Sigma H_\gamma)\times I)\cup (3$-handles).
\end{itemize}

In pratice, drawing the Kirby diagram for this handle decomposition is simple when there are no $3$-handles (i.\,e.\ $k-2p-b+1=0$; i.\,e.\ there are no parallel $\alpha,\gamma$ curves). Perform handle-slides on the $\alpha,\beta$ curves so that the pair $(\alpha,\beta)$ is standard. Then draw a $1$-handle for each cut arc in the $\alpha,\beta$ pages and for each parallel $\alpha,\beta$ curve; draw a $2$-handle for each $\gamma$ curve with framing given by the surface framing. See Figure \ref{fig:Kirbyexamples}.

In the above procedure, the $2$-handles, $3$-handles, and the $1$-handles corresponding to parallel $\alpha$ and $\beta$ curves should be familiar to a reader who has seen the construction of a Kirby diagram for a closed $4$-manifold from a trisection \cite{gaykirby}. The $1$-handles corresponding to cut arcs are more novel. These are apparent by actually considering the gluing $H_\alpha\cup_\Sigma H_\beta$. Take $H_\beta=\Sigma\times[0,1]\cup(2$-handles along $\beta\times 0)$ to be the standard compression body in $S^3$ (see Figure \ref{fig:kirbyexplain}). For each $\alpha$ curve dual to a $\beta$ curve, a $2$-handle can be glued to $\alpha\times 1$ in $S^3$. For each $\alpha$ curve parallel to a $\beta$ curve, we add a parallel dotted circle and attach a $2$-handle along $(\alpha\times 1)$ in $\#_{k-2p-b+1} S^1\times S^2$. This yields $H_\alpha\cup_\Sigma H_\beta\subset \#_{k-2p-b+1} S^1\times S^2$. 

The boundary of $H_\alpha\cup_\Sigma H_\beta$ consists of the $\alpha$ and $\beta$ pages  glued together along collars of their boundary, so $\boundary (H_\alpha\cup_\Sigma H_\beta)$ is a surface of genus-$(2p+b-1)$. Moreover, $(\#_{k+2p+b-1}S^1\times S^2)\setminus (H_\alpha\cup_\Sigma H_\beta)$ is a handlebody of genus-$(2p+b-1)$. The dotted circles obtained from the cut arcs of the $\alpha$ and $\beta$ pages are cores of this handlebody (see Figure \ref{fig:kirbyexplain}). Add these dotted circles to the diagram so that now $H_\alpha\cup_\Sigma H_\beta$ is embedded in $\#_{k}S^1\times S^2$, so that ($\#_{k}S^1\times S^2)\setminus (H_\alpha\cup_\Sigma H_\beta)$ is a product from the $\alpha$ page to the $\beta$ page. The total $\#_{k}S^1\times S^2$ is then exactly $\boundary X_1$. 

We take $X_2$ to be the trace of a cobordism from $H_\beta$ to $H_\gamma$; thicken the $\beta$ compression body and glue $2$-handles for each $\gamma$ curve that is dual to the $\beta$ curves (up to handle slides). By adding $2$-handle attaching circles parallel to these $\gamma$ curves (with framing equal to the surface framing), we can view $X_1$ and $X_2$ as living in the $4$-manifold described by this Kirby diagram. Finally, to include $X_3$ in this diagram we attach $3$-handles along essential spheres in $H_\alpha\cup_\Sigma H_\gamma$ (i.e. one $3$-handle for each $\gamma$ curve parallel to an $\alpha$ curve, up to handle slides.)

\begin{figure}
{\center{\includegraphics[height=30mm]{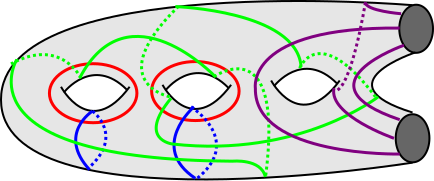}\hspace{.25in}\includegraphics[height=30mm]{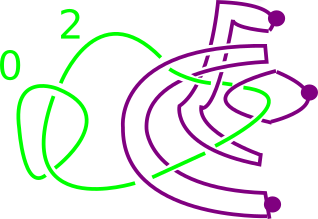}\\\vspace{.2in}\includegraphics[height=25mm]{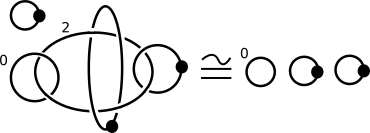}}}
\caption{Top left: a $(g,k,p,b)=(3,3,1,2)$-trisection diagram. Here, $\alpha$ and $\beta$ are standard. We have drawn a cut-system (purple) for the $\alpha$ and $\beta$ pages. Top right: We find a Kirby diagram for the pictured $4$-manifold. Each cut arc doubles to a $1$-handle curve; push one copy into the $\alpha$ compression body (``outside the surface'') and the other into the $\beta$-compression body (``inside the surface''). The $\gamma$ curves become $2$-handle attaching circles with framing given by the surface framing. There are no $3$- or $4$-handles. Bottom: We can now easily see that the pictured $4$-manifold is $(S^2\times D^2)\natural (S^1\times B^3)\natural (S^1\times B^3)$.}
\label{fig:Kirbyexamples}
\end{figure}

\begin{figure}
\includegraphics[width=\textwidth]{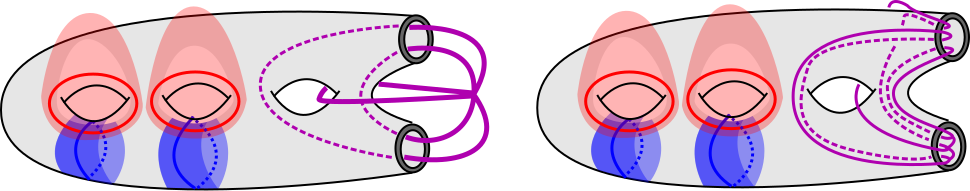}
\caption{We still consider the relative trisection of Figure \ref{fig:Kirbyexamples}. Left: $H_\alpha\cup_\Sigma H_\beta$ lies inside $\#_{k-2p-b+1} S^1\times S^2$. The complement of $H_\alpha\cup_\Sigma H_\beta$ is a handlebody; we have drawn the spine. Right: We isotope the spine to see cores of the handlebody $(\#_{k-2p-b+1} S^1\times S^2)\setminus(H_\alpha\cup_\Sigma H_\beta)$ as doubles of a cut system of the $\alpha$ and $\beta$ pages.}
\label{fig:kirbyexplain}
\end{figure}

Finding a relative trisection of $X^4$ from a Kirby diagram is more challenging. In \cite{trisectionfrompage} Castro, Gay, and Pinz\'on-Caicedo describe how to obtain a relative trisection from a Kirby diagram of $X^4$ and a page of an open-book on $\boundary X^4$ within the diagram.

\section{The Price Twist}\label{sec:price}
\subsection{Introduction}
\label{sec:pricetwist}
Let $S$ be a real projective plane embedded in a $4$-manifold $X^4$, with Euler number $e(S)=\pm 2$ (e.\,g.\ any $\RP^2$ in $S^4$). A tubular neighborhood of $S$ admits a handle structure consisting of a $0$-handle, a $1$-handle, and a $0$-framed $2$-handle running twice over the $1$-handle (Figure \ref{fig:pplusnbhd}). We call this tubular neighborhood $P_+$ or $P_-$, depending on the sign of $e(S)$. The boundary of $P$ is the quaternion space $Q$, named because it is the quotient of $S^3$ by the action of the quaternion group. A perhaps more useful description for low-dimensional topologists is that $Q$ is a Seifert-fibered space over $S^2$, with three exceptional fibers of index $\pm2,\pm2,\mp2$. Following \cite{katanaga}, we call these fibers $S_0,S_1,S_{-1}$. We do not specify the indices of these fibers, as they may be permuted by a homeomorphism of $Q$. See Figure \ref{fig:pplusnbhd} for an illustration of $S_0, S_1, S_{-1}$ in $Q=\boundary P_-$.

\begin{figure}
\includegraphics[height=35mm]{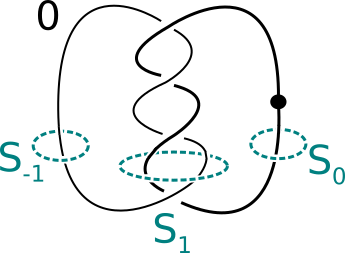}
\caption{A kirby diagram of $P_-$. (Mirror the diagram to obtain $P_+$). $Q=\boundary P_{\pm}$ is a Seifert fibered space with three singular fibers $S_0,S_1,S_{-1}$ as pictured. The $S_{-1}$ fiber bounds a disk in $P_{\pm}$.}
\label{fig:pplusnbhd}
\end{figure}

Note $\boundary (X^4\setminus \nu(S))\cong Q$. Label the singular fibers in $\boundary(X^4\setminus\nu(S))$ so that the trivial regluing of $\nu(S)$ to obtain $X^4$ corresponds to the map $\boundary(\nu(S))\to\boundary(X^4\setminus S)$ given by $(S_1,S_0,S_{-1})\mapsto(S_1,S_0,S_{-1})$

Price \cite{pricepaper} has classified the self-homeomorphisms of $Q$, finding that there are six up to isotopy. These maps preserve the Seifert fiber structure, and are determined simply by the induced permutation of the singular fibers. Moreover, Price showed that the map that permutes $S_0,S_1$ extends over $P_{\pm}$. Therefore, there are at most three $4$-manifolds (up to diffeomorphism) that may arise from deleting $\nu(S)$ from $X^4$ and regluing according to $\phi:\boundary\nu(S)\to\boundary (X^4\setminus\nu(S))$. These $4$-manifolds are:

\begin{itemize}
\item $X^4$, when $\phi(S_{-1})=S_{-1}$.
\item $\tau_S(X^4)$, when $\phi(S_{-1})=S_{0}$. Using the Mayer-Vietoris sequence, we see $H_1(\tau_S(X^4))\neq H_1(X^4)$.
\item $\Sigma_S (X^4)$, when $\phi(S_{-1})=S_1$.
\end{itemize}

We call $\Sigma_S(X^4)$ the Price twist of $X^4$ along $S$. In the notation of Akbulut and Yasui \cite{plugs}, 
this operation is equivalent to a certain plug twist. See \cite{plugs}, \cite{akbulutrp2} for more on this point of view. For our purposes, it is enough to notice that the Price twist is a way of constructing potentially exotic $4$-manifolds, and is most interesting in the case $X^4=S^4$. When $X^4=S^4$, $\Sigma_S(S^4)$ is a homotopy $4$-sphere. It is unknown under which conditions $\Sigma_S(X^4)$ is diffeomorphic to $X^4$. Katanaga et.\ al.\ \cite{katanaga} showed that this operation generalizes the Gluck twist: If $K$ is a $2$-sphere smoothly embedded in $X^4$ with trivial normal bundle, and $P$ is a unknotted real projective plane in a $4$-ball disjoint from $K$, then the Gluck twist of $X^4$ along $K$ is diffeomorphic to $\Sigma_{K\# P}(X^4)$. In particular, this means that in some cases $\Sigma_S(X^4)\not\cong X^4$ (see for example \cite{akbulut}), but there are of course no known examples of this phenomenon in $S^4$. (In fact, the authors are unaware of any examples of this phenomenon in any orientable $4$-manifold. The surgeries displayed in \cite{akbulut} which change the smooth structure of the ambient $4$-manifold all take place in non-orientable $4$-manifolds.) Whether the Price twist strictly generalizes the Gluck twist in $S^4$ is related to the \emph{Kinoshita conjecture}.

\begin{question}[Kinoshita]
Given a real projective plane $S$ smoothly embedded in $S^4$, can $S$ be decomposed as $K\# P$ for some $2$-sphere $K$ and an unknotted real projective plane $P$?
\end{question}

The answer to the above question is known to be ``yes'' in some cases; see e.\,g.\ \cite{kamada}. One might study the Kinoshita conjecture by understanding the basic algebraic topology of $S^4\setminus \nu(S)$. Suppose $S=K\# P$ for a $2$-knot $K$ and unknotted real projective plane $P$. Then $S^4\setminus\nu(S)\cong (S^4\setminus\nu(K))\natural (S^4\setminus\nu(P))$, so $\pi_1(S^4\setminus\nu(S))\cong \pi_1(S^4\setminus K\mid [\gamma]^2=0)$, where $[\gamma]$ is a normal generator of $\pi_1(S^4\setminus K)$. Fundamental groups of $2$-knot complements have been classified as admitting certain presentations by Kamada \cite{2knotgroups}, so $\pi_1(S^4\setminus\nu(S))\cong \pi_1(\tau_S(S^4))$ is a potential obstruction to $S$ admitting the decomposition $K\#P$.

\subsection{Two preferred trisections of $P_{\pm}$}\label{sec:prefer}
\label{sec:preferredtrisections}
In this section, we describe two particular $(g,k,p,b)=(2,2,0,3)$-relative trisections of $P_{+}$ (the mirror images of which are naturally trisections of $P_-$), shown in Figure \ref{fig:trisection1and2}. Call these trisections $T_1$ and $T_2$. From $T_i$, the dicussion in Section \ref{sec:getKirby} yields a Kirby diagram, from which we can check that the trisected manifold is indeed $P_+$ (see Figure \ref{fig:checkKirby}). The three singular fibers of $Q=\boundary P_+$ form the bindings of the open book induced by $T_i$ on $Q=\boundary P_+$, and each page has monodromy consisting of two right-handed Dehn twists around two boundary components and two left-handed Dehn twists around the third. In $T_1$, the twists around the $S_{-1}$ boundary are left-handed. In $T_2$, the twists around the $S_{-1}$ boundary are right-handed (See Figs. \ref{fig:monodromy1},\ref{fig:monodromy2}).

\begin{figure}
\includegraphics[height=50mm]{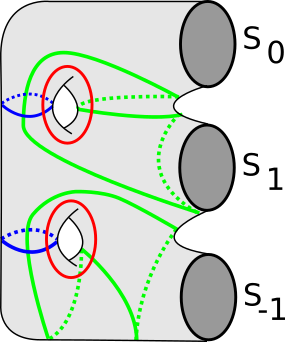}\hspace{.5in}\includegraphics[height=50mm]{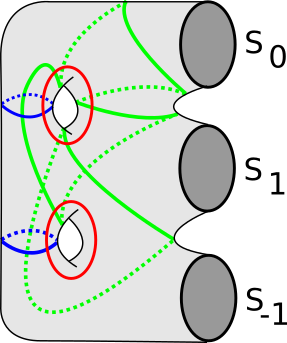}
\caption{Left: The relative trisection $T_1$ of $P_+$. Right: the relative trisection $T_2$ of $P_+$. These are both $(g,k,p,b)=(2,2,0,3)$-relative trisections.}
\label{fig:trisection1and2}
\end{figure}

\begin{figure}
\includegraphics[width=40mm]{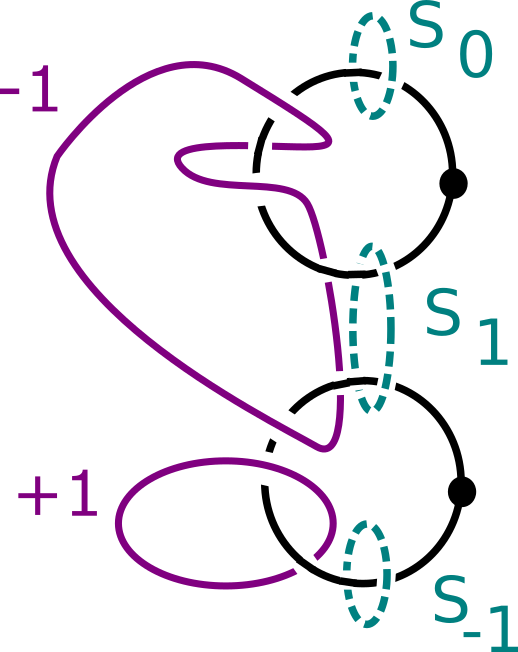}\hspace{.5in}\includegraphics[width=40mm]{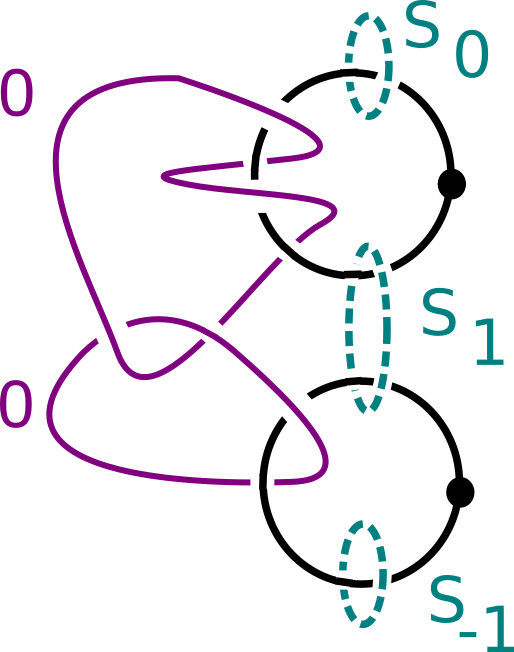}
\caption{Left: Kirby diagram obtained from $T_1$. Right: Kirby diagram obtained from $T_2$. Neither diagram has any $3$- or $4$-handles. Both depict $P_+$. The $S_i$ are the bindings of the open book on $Q$ (shown here after some handle slides).}
\label{fig:checkKirby}
\end{figure}

\begin{figure}
\includegraphics[height=50mm]{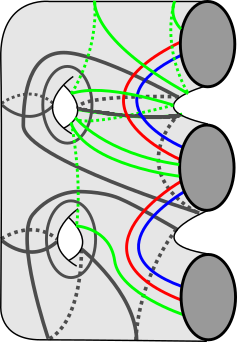}\hspace{.5in}\includegraphics[height=50mm]{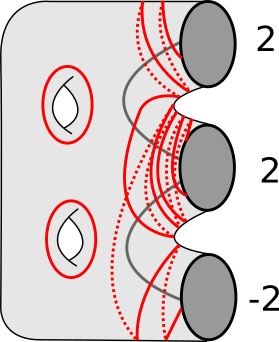}
\caption{We perform the algorithm of \cite{monodromypaper} to find the monodromy of the open book induced on $Q$ by $T_1$. Left: resulting cut systems for each page $X_i\cap X_j\cap\boundary P_+$. Right: the effect of the monodromy on the $X_1\cap X_2\cap\boundary P_+$ system. The monodromy automorphism consists of two right-handed Dehn twists around each of $S_0,S_1$ and two left-handed twists around $S_{-1}$.}
\label{fig:monodromy1}
\end{figure}
\begin{figure}
\includegraphics[height=50mm]{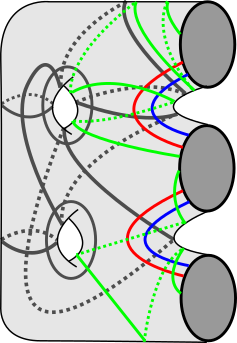}\hspace{.5in}\includegraphics[height=50mm]{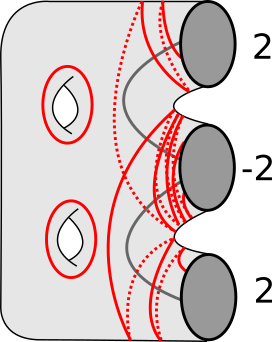}
\caption{We perform the algorithm of \cite{monodromypaper} to find the monodromy of the open book induced on $Q$ by $T_2$. Left: resulting cut systems for each page $X_i\cap X_j\cap\boundary P_-$. Right: the effect of the monodromy on the $X_1\cap X_2\cap\boundary P_+$ system. The monodromy automorphism consists of two right-handed Dehn twists around each of $S_0,S_{-1}$ and two left-handed twists around $S_{1}$.}
\label{fig:monodromy2}
\end{figure}


\section{Relative trisections of surface complements}\label{sec:complement}
\label{sec:trisectcomplement}
In this section, we show how to produce a relative trisection of the complement of a surface $S$ in a $4$-manifold $X^4$. In particular, we can produce a $(g,k,p,b)=(g,k,0,3)$-trisection of any $\RP^2$ complement in $X^4$. 
For notational ease, we will often use the shorthand ``$\cap\boundary$'' to denote ``$\cap\boundary(X^4\setminus\nu(S))$''.

Naively, one might attempt to construct a trisection of a surface complement in the following (generally incorrect) way.
\begin{itemize}
    \item Say $X^4$ is a $4$-manifold with trisection $(X_1,X_2,X_3)$. 
    \item Let $S\subset X^4$ be a surface in $(c,b)$-bridge position with respect to the trisection $(X_1,X_2,X_3)$.
    \item Delete a tubular neighborhood of $S$ from each $X_i$ to find a trisection on $X^4\setminus\nu(S)$.
\end{itemize}

This procedure is successful exactly when $S\cong S^2$ and $c=1$ (as in \cite{jeffgluck}). Otherwise, this procedure {\emph{does not}} yield a relative trisection of $X^4 \setminus \nu(S)$ for (in part) the following reason. 
Recall from Definition \ref{relativedefinition} that a relative trisection must induce an open book on the $3$-dimensional boundary, with pages $X_i\cap X_j\cap\boundary$. In this setting, $((X_i\cap X_j)\setminus\nu(S))\cap\boundary(X^4\setminus\nu(S))\cong \sqcup_b (S^1\times I)$, which is a collection of $b$ annuli corresponding to the bridges of $S$ in $X_i\cap X_j$. Note that if $b=1$, then $S$ is a sphere, since $\chi(S)=3c-b$. If $b>1$, then $((X_i \cap X_j)\setminus\nu(S))\cap\boundary(X^4\setminus\nu(S))$ is disconnected, and certainly cannot be a page in an open book on $\boundary(X^4\setminus(S))$.

To deal with this problem, we introduce an operation on manifolds $Y^4$ divided into three pieces   $Y^4=Y_1\cup Y_2\cup Y_3$ (not necessarily as a relative trisection, but still assuming $Y_i\cap Y_j=\boundary Y_i\cap\boundary Y_j$) that can reduce the number of components of $Y_i\cap Y_j\cap\boundary Y^4$.

\begin{definition}\label{def:bdystab}
Let $Y^4$ be a $4$-manifold with nonempty boundary. Let $Y^4 =Y_1\cup Y_2\cup Y_3$, where $Y_i\cap Y_j=\boundary Y_i\cap\boundary Y_j$. Let $C$ be an arc properly embedded in $Y_i\cap Y_j\cap\boundary Y^4$ with endpoints in $Y_1\cap Y_2\cap Y_3$. Let $\nu(C)$ be a fixed open tubular neighborhood of $C$. Let
\begin{itemize}
\item $\newtilde{Y}_i:=Y_i\setminus\nu(C)$,
\item $\newtilde{Y}_j:= Y_j\setminus\nu(C)$,
\item $\newtilde{Y}_k:=Y_k\cup\overline{\nu(C)}$.
\end{itemize}

We refer to the replacement of the triple $(Y_1,Y_2,Y_3)$ by $(\newtilde{Y}_1,\newtilde{Y}_2,\newtilde{Y}_3)$ as a \emph{boundary-stabilization}. We say that we have boundary-stabilized $Y_k$.

This move is depicted in Figure \ref{fig:neighborhoodarc}.
\end{definition}

We illustrate the effect of this stabilization on each of $Y_1, Y_2, Y_3$, double intersection, and triple intersection by studying trisections of $\RP^2$ complements in the following subsection. In Section \ref{sec:general}, we will use the same principle to trisect the complement of an arbitrary surface complement.

\subsection{Trisecting complements of embedded $\RP^2$s} \label{sec:rp2subsection}

Let $(X_1,X_2,X_3)$ be a $(g,k)$-trisection of $X^4$. Let $S$ be an embedded $\RP^2$ in $X$. By \cite{newsurfacetrisections}, $S$ can be isotoped to be in bridge position with respect to $(X_1,X_2,X_3)$. 
We stabilize the trisection as in \cite{newsurfacetrisections} so that $S\cap X_i=D^2$ for each $i$. 
Then $S\cap X_i\cap X_j$ is a trivial $2$-bridge tangle. 

Delete a tubular neighborhood of $S$ from $X^4$; let $X'_i:=X_i\setminus\nu(S)$.  For each $i=1,2,3$, with $\{i,j,k\}=\{1,2,3\}$, let $C_i$ be an arc in $X'_j\cap X'_k\cap\boundary (X^4\setminus\nu(S))$ (with endpoints in $X'_1\cap X'_2\cap X'_3)$ which meets two different boundary components of $X'_1\cap X'_2\cap X'_3$. Take $C_1,C_2,C_3$ to have disjoint endpoints, and to twist zero times around the $3$-dimensional tubes of $\nu(S)\cap (X_i\cap X_j)$ (i.e. we ask $S\cap (X_i\cap X_j)$ cobounds disks $D$ in $X_i\cap X_j$ with arcs in $X_i\cap X_j\cap X_k$ so that $C_k\cap D=\emptyset$). This will not matter until we attempt to find the resulting relative trisection diagram. In principal, one could allow $C_i$ to twist around $\nu(S)\cap(X_j\cap X_k)$ arbitrarily and modify the relative trisection accordingly, but for simplicitly we restrict the intersection). We will boundary-stabilize each $X_i$ along $C_i$ to find a relative trisection of $X^4 \setminus \nu(S)$. 

The boundary-stabilization move of $X'_3$ along $C:=C_3$ is pictured in Figures \ref{fig:setuparcs} and \ref{fig:neighborhoodarc}. For ease of discussion, we write $\overline{\nu(C)}=[-\epsilon,1+\epsilon]\times B$, where $B$ is a (closed) half $3$-ball. Write $\boundary B=D_+\cup D_-$, where $D_+$ and $D_-$ are each disks, and $\nu(C)\cap\boundary (X^4\setminus\nu(S))=[-\epsilon,1+\epsilon]\times D_-$. Write $B$ as the union of two quarter (closed) $3$-balls $B_1$ and $B_2$ along a disk so that $\overline{\nu(C)}\cap X'_1=I\times B_1$, $\overline{\nu(C)}\cap X'_2=I\times B_2$, and $\overline{\nu(C)}\cap X'_3=([-\epsilon,0]\cup[1,1+\epsilon])\times B$. See Figure \ref{fig:halfball}.

\begin{figure}
\includegraphics[width=.3\textwidth]{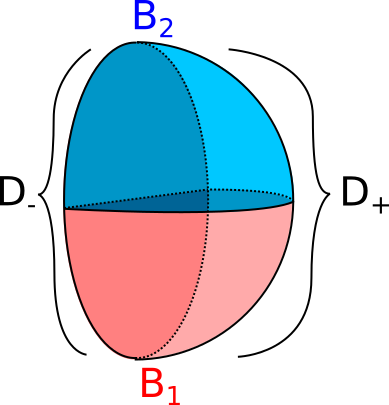}
\caption{To understand the effect of boundary-stabilizing along arc $C$ in $X'_1\cap X'_2\cap\boundary$, we write $\overline{\nu(C)}$ as $[-\epsilon,1+\epsilon]\times B$, where $B$ is a closed half $3$-ball. Here, we draw $B$. We decompose $B$ into two quarter-balls $B_1$ and $B_2$, so that $\overline{\nu(C)}\cap X'_1=I\times B_1$ and $\overline{\nu(C)}\cap X'_2=I\times B_2$. The boundary of $B$ is decomposed into two disks, $D_-$ and $D_+$ (which both meet $B_1$ and $B_2$) so that $\overline{\nu(C)}\cap\boundary=[-\epsilon,1+\epsilon]\times D_-$.}\label{fig:halfball}
\end{figure}

This boundary-stabilization increases the genus of $X'_3$, $X'_1\cap X'_3$ and $X'_2\cap X'_3$ each by one. We illustrate the effect on the $X'_i\cap X'_j\cap\boundary (X^4\setminus\nu(S))$ and $X'_1\cap X'_2\cap X'_3$ in Figures \ref{fig:pageeffect} and \ref{fig:tripleintersection}; each of these figures relate to a schematic triplane diagram. We discuss the topology of each trisection piece before and after the boundary-stabilization move in greater detail in the following paragraphs. (Many of these paragraphs are repetitive due to the symmetry of a trisection, but we consider each piece separately for clarity.) Recall the notation ``$\cap\boundary$'' means ``$\cap\boundary(X^4\setminus\nu(S))$''.

\begin{figure}
\includegraphics[width=60mm]{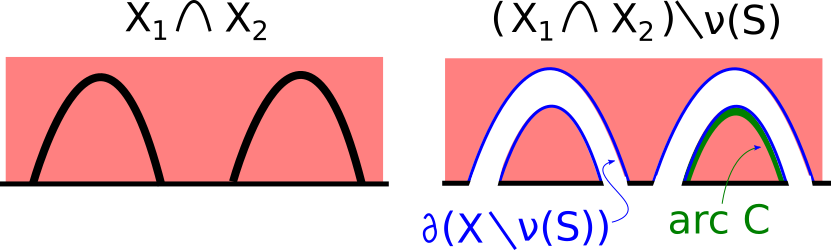}
\caption{Left: $S\cap X_1\cap X_2$. Right: We find an arc $C$ in $X'_1\cap X'_2\cap\boundary$ which meets two different boundary components of $X'_1\cap X'_2\cap X'_3$ [i.\,e.\ an arc that runs along one bridge.]}
\label{fig:setuparcs}
\end{figure}

\begin{figure}
\includegraphics[width=\textwidth]{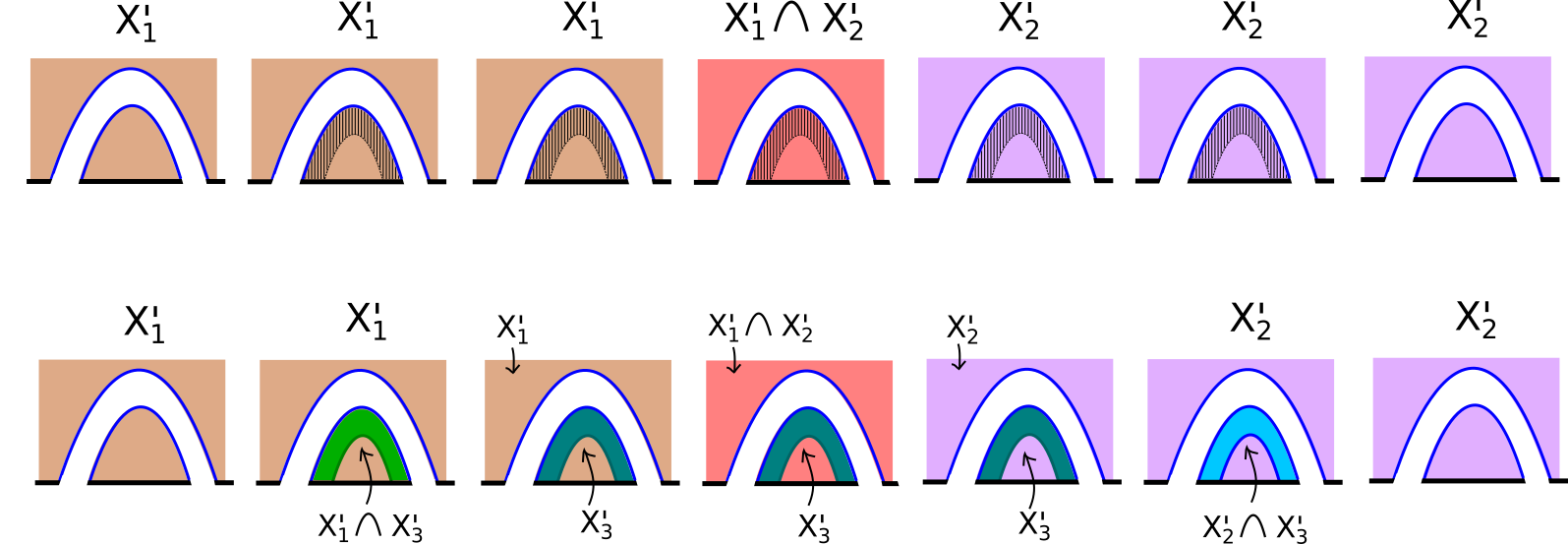}
\caption{Top: The shaded regions are slices of a neighborhood of $C$ in $X^4\setminus\nu(S)$ before stabilizing. Bottom: To boundary-stabilize, we declare this neighborhood is in $X'_3$.}
\label{fig:neighborhoodarc}
\end{figure}

\begin{figure}
\includegraphics[width=\textwidth]{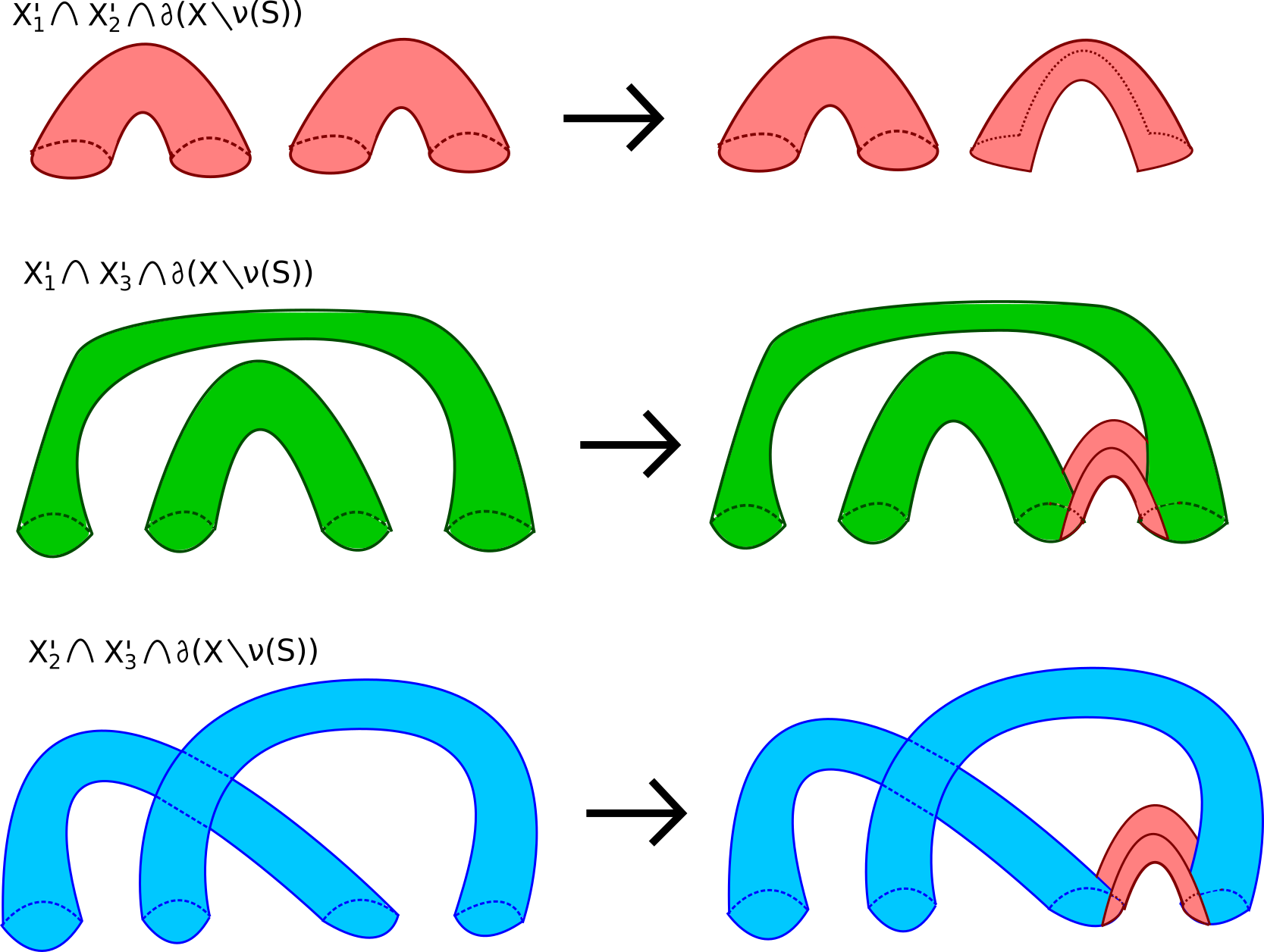}
\caption{The boundary-stabilization move deletes a band from $X'_1\cap X'_2\cap\boundary (X^4 \setminus\nu(S))$, and adds a band to each of $X'_1\cap X'_3\cap\boundary$ and $X'_2\cap X'_3\cap\boundary$. These are not literally three instances of the same band. The band deleted from $X'_1\cap X'_2\cap\boundary$ is $I\times (B_1\cap B_2\cap D_-)$. The band added to $X'_1\cap X'_3\cap\boundary$ is $I\times(B_1\cap D_+\cap D_-)$. The band added to $X'_2\cap X'_3$ is $I\times(B_2\cap D_+\cap D_-)$.}
\label{fig:pageeffect}
\end{figure}

\begin{figure}
\includegraphics[width=\textwidth]{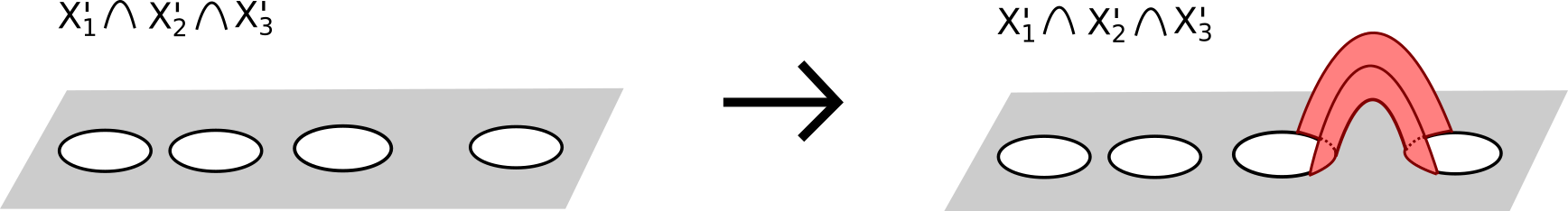}
\caption{The boundary-stabilization move adds a band  $I\times (B_1\cap B_2\cap B_+)$ to the triple intersection $X'_1\cap X'_2\cap X'_3$.}
\label{fig:tripleintersection}
\end{figure}

{$\mathbf{X'_1.}$} Before boundary-stabilizing $X'_3$, $X'_1\cong\natural_{k+1} S^1\times B^3$. After boundary-stabilizing $X'_3$, $X'_1\cong\natural_{k+1} S^1\times B^3$. From $X'_1$, in the boundary-stabilization we carve out a neighborhood of arc $C$ in $\boundary X'_1$. This does not change the topology of $X'_1$.

{$\mathbf{X'_2.}$} Before boundary-stabilizing $X'_3$, $X'_2\cong\natural_{k+1} S^1\times B^3$. After boundary-stabilizing $X'_3$, $X'_2\cong\natural_{k+1} S^1\times B^3$. From $X'_2$, in the boundary-stabilization we carve out a neighborhood of arc $C$ in $\boundary X'_2$. This does not change the topology of $X'_2$.

{$\mathbf{X'_3.}$} Before boundary-stabilizing $X'_3$, $X'_1\cong\natural_{k+1} S^1\times B^3$. After boundary-stabilizing $X'_3$, $X'_3\cong\natural_{k+2} S^1\times B^3$. To $X'_3$, we add a neighborhood of arc $C$ (recall $\boundary C\in\boundary X'_3$, $\int C\cap X'_3=\emptyset$). The effect is to add another boundary-sum $S^1\times B^3$ component to $X'_3$.

{$\mathbf{X'_1\cap X'_2.}$} Before boundary-stabilizing $X'_3$, $X'_1\cap X'_2\cong \natural_{g+2} S^1\times B^2$. After boundary-stabilizing $X'_3$, $X'_1\cap X'_2\cong \natural_{g+2} S^1\times B^2$. From $X'_1\cap X'_2$,  in the boundary-stabilization we carve out a neighborhood of arc $C$ in $\boundary X'_1\cap X'_2$. This does not change the topology of $X'_1\cap X'_2$.

{$\mathbf{X'_2\cap X'_3.}$} Before boundary-stabilizing $X'_3$, $X'_2\cap X'_3\cong \natural_{g+2} S^1\times B^2$. After boundary-stabilizing $X'_3$, $X'_2\cap X'_3\cong \natural_{g+3} S^1\times B^2$. Then to $X'_2\cap X'_3$, in the boundary-stabilization we add $I\times B_2$. The effect is to add another boundary-sum $S^1\times B^2$ component to $X'_2\cap X'_3$.

{$\mathbf{X'_1\cap X'_3.}$} Before boundary-stabilizing $X'_3$, $X'_1\cap X'_3\cong \natural_{g+2} S^1\times B^2$. After boundary-stabilizing $X'_3$, $X'_1\cap X'_3\cong \natural_{g+3} S^1\times B^2$. To $X'_1\cap X'_3$, in the boundary-stabilization we add $I\times B_1$. The effect is to add another boundary-sum $S^1\times B^2$ component to $X'_1\cap X'_3$.

{$\mathbf{X'_1\cap X'_2\cap X'_3.}$} Before boundary-stabilizing $X'_3$, $X'_1\cap X'_2\cap X'_3\cong \Sigma_{g,4}$, a genus-$g$ surface with $4$ open disks deleted. After boundary-stabilizing $X'_3$, $X'_1\cap X'_2\cap X'_3\cong \Sigma_{g,4}\cup($a band) 
To $X'_1\cap X'_2\cap X'_3$, we add the band $I\times( B_1\cap B_2\cap D_+)$. In thiscase, we have assumed that the band meets two distinct boundary components of $\Sigma_{g,4}$, so the attachment yields $\Sigma_{g+1,3}$. Note now that $C_1$ or $C_2$ may meet only one boundary component of $\Sigma_{g+1,3}$.

{$\mathbf{X'_1\cap\boundary.}$} Before boundary-stabilizing $X'_3$, we have $X'_1\cap\boundary\cong S^1\times B^2$. After boundary-stabilizing $X'_3$, $X'_1\cap\boundary \cong S^1\times B^2$. From $X'_1\cap\boundary $, in the boundary-stabilization we carve out a neighborhood of arc $C$ in $\boundary (X'_1\cap\boundary)$. This does not change the topology of $X'_1\cap\boundary$.

{$\mathbf{X'_2\cap\boundary.}$} Before boundary-stabilizing $X'_3$, we have $X'_2\cap\boundary \cong S^1\times B^2$. After boundary-stabilizing $X'_3$, $X'_2\cap\boundary \cong S^1\times B^2$. From $X'_2\cap\boundary$, in the boundary-stabilization we carve out a neighborhood of arc $C$ in $\boundary (X'_2\cap\boundary)$. This does not change the topology of $X'_2\cap\boundary$.

{$\mathbf{X'_3\cap\boundary.}$} Before boundary-stabilizing $X'_3$, we have $X'_3\cap\boundary\cong S^1\times B^2$. After boundary-stabilizing $X'_3$, $X'_3\cap\boundary\cong\natural_{2} S^1\times B^2$. To $X'_3\cap\boundary$, we add $\nu(C)\cap\boundary=($arc parallel to $C)\times D^2$ (recall $C\subset X'_1\cap X'_2\cap \boundary X$). The effect is to add another boundary-sum $S^1\times B^2$ component to $X'_3\cap\boundary$.

{$\mathbf{X'_1\cap X'_2\cap\boundary.}$} Before boundary-stabilizing $X'_3$, $X'_1\cap X'_2\cap\boundary \cong \sqcup_2 S^1\times I=($two annuli). After boundary-stabilizing $X'_3$, $X'_1\cap X'_2\cap\boundary\cong D^2\sqcup (S^1\times I)$.From $X'_1\cap X'_2\cap\boundary$,  in the boundary-stabilization we carve out a neighborhood of arc $C$, which is a cocore of one annulus.

{$\mathbf{X'_2\cap X'_3\cap\boundary.}$} Before boundary-stabilizing $X'_3$, $X'_2\cap X'_3\cap\boundary\cong \sqcup_2 S^1\times I=$(two annuli). After boundary-stabilizing $X'_3$, $X'_2\cap X'_3\cap\boundary \cong(\sqcup_2 S^1\times I)\cup($band$)\cong \Sigma_{0,3}$. To $X'_2\cap X'_3\cap\boundary$, we add $I\times(B_2\cap D_+\cap D_-)$.

{$\mathbf{X'_1\cap X'_3\cap\boundary.}$} Before boundary-stabilizing $X'_3$, $X'_1\cap X'_3\cap\boundary\cong\sqcup_2 S^1\times I=$(two annuli). After boundary-stabilizing $X'_3$, $X'_1\cap X'_3\cap\boundary \cong(\sqcup_2 S^1\times I)\cup($band$)\cong \Sigma_{0,3}$. To $X'_1\cap X'_3\cap\boundary$, we add $I\times(B_1\cap D_+\cap D_-)$.

{$\mathbf{X'_1\cap X'_2\cap X'_3\cap\boundary.}$} Before boundary-stabilizing $X'_3$, $X'_1\cap X'_2\cap X'_3\cap\boundary\cong\sqcup_4 S^1$, the boundary components of $X'_1\cap X'_2\cap X'_3$. After boundary-stabilizing $X'_3$, $X'_1\cap X'_2\cap X'_3\cap\boundary \cong \sqcup_3 S^1$. During the boundary-stabilization, we surger $X'_1\cap X'_2\cap X'_3\cap\boundary$ along the two intervals in $\nu(C)$. The band we surger along is $I\times(B_1\cap B_2\cap D_-)$; we delete neighborhoods of $\boundary I\times (B_1\cap B_2\cap D_-)$ and glue in the remaining boundary. We assumed that the endpoints of $C$ met different boundary components of $X'_1\cap X'_2\cap X'_3$, which causes the number of components of $X'_1\cap X'_2\cap X'_3\cap\boundary$ to decrease after the boundary-stabilization. If we did the same procedure on an arc with both endpoints on one boundary component then after the boundary-stabilization we would have had $X'_1\cap X'_2\cap X'_3\cap \boundary=\sqcup_5 S^1$. 
This concludes the analysis of the effect of boundary-stabilizing $X'_3$ along $C=C_3$.

Now similarly boundary-stabilize $X'_1$ along $C_1$ and $X'_2$ along $C_2$ to obtain $X^4\setminus\nu(S)=\newtilde{X}_1\cup\newtilde{X}_2\cup\newtilde{X}_3$ (where $\newtilde{X}_i$ is obtained from $X'_i$ by performing all three boundary stabilizations). The end result will have $\newtilde{X}_i\cap \newtilde{X}_j\cap\boundary(X^4\setminus\nu(S))=$(thrice punctured sphere) or (once punctured torus), depending on our choice of arcs $C_1, C_2$, and $C_3$ (see Figure \ref{fig:pageinboundary}).

One point of these boundary-stabilizations is to ensure that $\newtilde{X}_i\cap \newtilde{X}_j\cap\boundary (X^4\setminus\nu(S))$ is connected. We must furthermore check that $\newtilde{X}_i \cap\boundary(X^4\setminus\nu(S))$ is a product $(\newtilde{X}_i\cap \newtilde{X}_j\cap\boundary (X^4\setminus\nu(S)))\times I$ and also a product $(\newtilde{X}_i\cap \newtilde{X}_k\cap\boundary (X^4\setminus\nu(S)))\times I$ (and that these two product structures agree), so that there is an induced open book on the boundary of $X^4\setminus\nu(S)$.

\begin{claim}
 $\newtilde{X}_1\cap\boundary(X^4\setminus \nu(S))$ is a product over $\newtilde{X}_1\cap \newtilde{X}_2\cap\boundary(X^4\setminus \nu(S))$ and over $\newtilde{X}_1\cap \newtilde{X}_3\cap\boundary (X^4\setminus\nu(S))$. Moreover, these product structures agree.
\end{claim}
 
 \begin{proof}
 See Figure \ref{fig:pageinboundary} for an illustration of this proof.

Before any boundary-stabilizations, $X'_1\cap\boundary$ is a solid torus. $X'_1\cap X'_2\cap\boundary$ and $X'_1\cap X'_3\cap\boundary$ are each disjoint unions of two annuli. The core of each annulus is a longitude of the solid torus $X'_1\cap\boundary$. 
 
From the discussion so far of Section \ref{sec:rp2subsection}, recall that after performing the $X'_3$-boundary stabilization we have
\begin{align*}
X'_1\cap\boundary&\cong S^1\times B^2,\\
X'_1\cap X'_2\cap\boundary&\cong D^2\sqcup (S^1\times I),\\
X'_1\cap X'_3\cap\boundary&\cong \sqcup_2 (S^1\times I)\cup(\text{band between the two components})\cong\Sigma_{0,3}.
\end{align*}
Similarly, after further performing the $X'_2$-boundary stabilization we have
\begin{align*}
X'_1\cap\boundary&\cong S^1\times B^2,\\
X'_1\cap X'_2\cap\boundary&\cong D^2\sqcup (S^1\times I)\cup(\text{band between the two components})\\&\cong S^1\times I,\\
X'_1\cap X'_3\cap\boundary&\cong \Sigma_{0,3}\setminus\nu(\text{arc between two boundary components})\cong S^1\times I.
\end{align*}

Thus, after performing the $X'_2$ and $X'_3$ boundary-stabilizations, $X'_1\cap\boundary$ is a solid torus and each of $X'_1\cap X'_2\cap\boundary$ and $X'_1\cap X'_3\cap\boundary$ is an annulus on $\boundary (X'_1\cap\boundary)$ whose core is a longitude of $X'_1\cap\boundary$. At this point, we have a product structure $X'_1\cap\boundary \cong (S^1\times I\times I)$, where $X'_1\cap X'_2\cap\boundary=S^1\times I\times 0$ and $X'_1\cap X'_3\cap\boundary=S^1\times I\times 1$.

The $X'_1$ boundary-stabilization increases the genus of $X'_1\cap\boundary$. On the solid tube added to $X'_1\cap\boundary $ to obtain $\newtilde{X}_1\cap\boundary$, there is a band on the boundary (parallel to the core of the tube) contained in $\widetilde{X}_1\cap \widetilde{X}_2\cap\boundary$ and another band on the boundary (parallel to the core of the tube) contained in $\widetilde{X}_1\cap \widetilde{X}_3\cap\boundary$. The effect on the product structure of $X'_1\cap\boundary$ is to add a product tube (band)$\times I$.\end{proof}

The above claim holds similarly for $\newtilde{X}_2\cap\boundary$ and $\newtilde{X}_3\cap\boundary$, exchanging the roles of $\newtilde{X}_1,\newtilde{X}_2,$ and $\newtilde{X}_3$. Thus, $(\newtilde{X}_1,\newtilde{X}_2,\newtilde{X_3})$ induces an open-book structure on $\boundary(X^4\setminus \nu(S))$, so  $(\newtilde{X}_1,\newtilde{X}_2,\newtilde{X_3})$  is a relative trisection of $X^4\setminus\nu(S)$.

\begin{figure}
\includegraphics[width=\textwidth]{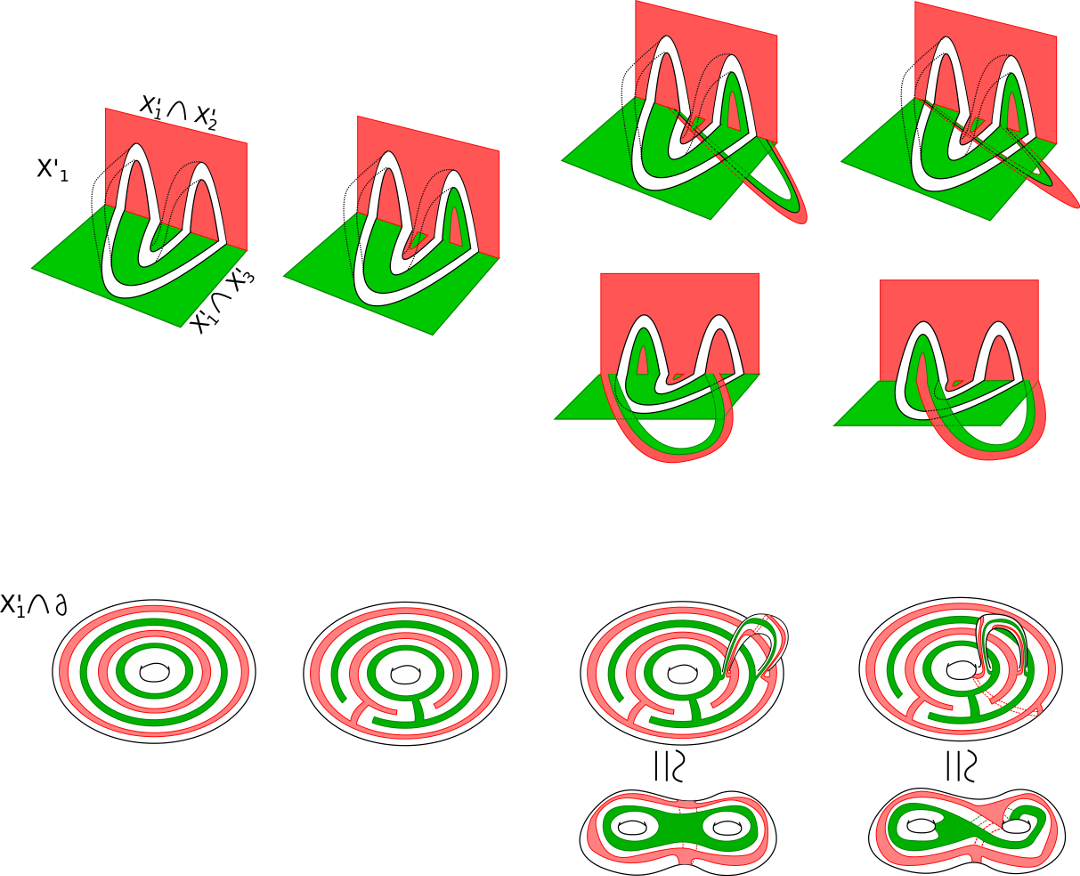}
\caption{Top left: A schematic of $X'_1$. To the right, we draw $X'_1$ after boundary-stabilizing $X'_2$ and $X'_3$. Top two, third column: two perspectives of $\newtilde{X}_1$ after boundary-stabilizing $X'_1$. Top two, rightmost: two perspectives of $\newtilde{X}_1$ after a boundary-stabilizing $X'_1$ (with a different choice of $C_1$). Bottom left: Before boundary-stabilizing, $X'_1\cap\boundary$ is a solid torus. On its boundary, $X'_1\cap X'_2\cap\boundary$ and $X'_1\cap X'_3\cap\boundary$ are each two longitudinal annuli. Bottom left, second picture: $X'_1\cap\boundary$ after boundary-stabilizing $X'_2$ and $X'_3$. Bottom left, two rightmost pictures: $\newtilde{X}_1\cap\boundary$ after boundary-stabilizing $X'_1$, for two different choices of $C_1$ (corresponding to the images in the top row).}
\label{fig:pageinboundary}
\end{figure}

In conclusion, we now show how to find a relative trisection diagram $(\Sigma',\alpha',\beta',\gamma')$ of this relative trisection. See Figure \ref{fig:triplanediagram} for several simple examples. We start with a shadow diagram $(\Sigma,\alpha,\beta,\gamma, s_\alpha,s_\beta,s_\gamma)$ of $S$ (recall Definition \ref{shadowdef}). Stabilize $(\Sigma,\alpha,\beta,\gamma)$ as necessary so that we may take $s_\alpha,s_\beta$, and $s_\gamma$ each to be two disjoint arcs. Identify $\Sigma$ with $X_1\cap X_2\cap X_3$. The arcs $C_1$, $C_2$, and $C_3$ are parallel to one arc of $s_\beta,s_\gamma,s_\alpha$, respectively (with correct framing, since we assumed that $C_k$ does not twist around $S\cap X_i\cap X_j$ in $X_i\cap X_j$). 

Obtain $\Sigma'$ from $\Sigma$ by deleting an open neighborhood of $\boundary s_*$ and attaching an orientation-preserving band for each $C_i$, with endpoints around $\boundary C_i$. The core of $C_3$ and its shadow (an arc in $s_\alpha$) together bound a disk in $\newtilde{X}_1\cap \newtilde{X}_2$, giving an $\alpha'$ curve (shadow of $C_3)\cup($core of band corresponding to $C_3$). The other $\alpha'$ curve encircles the shadow of $C_3$ in $\Sigma$. Similar holds for the $\beta'$ and $\gamma'$ curves. See Figure \ref{fig:triplanediagram} for several small examples of relatively trisecting $X^4\setminus\nu(S)$.

\begin{figure}
\includegraphics[width=\textwidth]{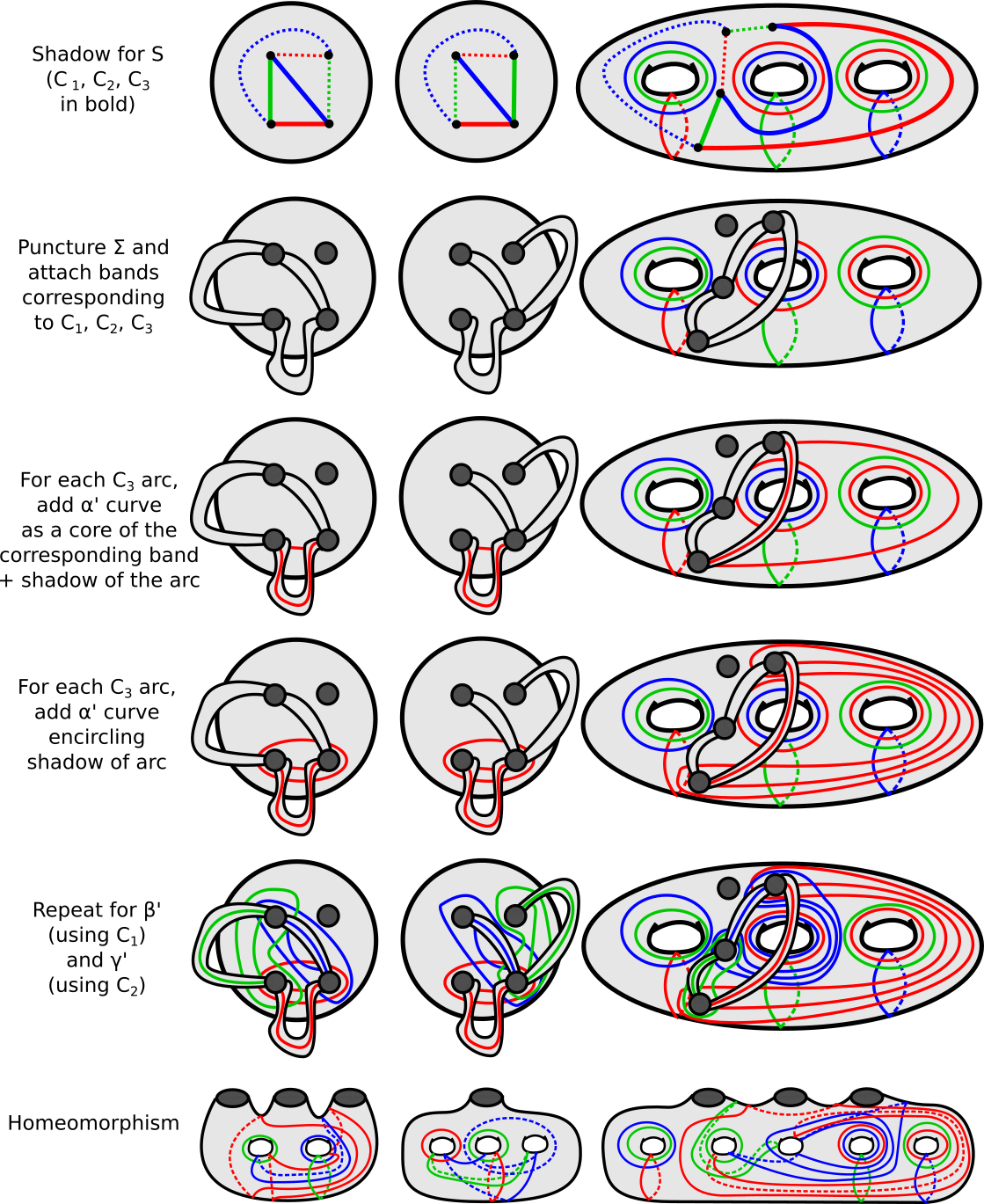}
\caption{Illustration of the process for finding a relative trisection of $X^4\setminus\nu(S)$, where $S\cong\RP^2$. In the top row, we draw possible shadow diagrams for $S$.  Down each column, we show the process in finding the relative trisection, where the bold shadows are parallel to the chosen arcs $C_1$, $C_2$, and $C_3$.  The final relative trisection is in the 5th row; we give an equivalent diagram (related by a surface automorphism) in the 6th row.  
Left to right, the final relative trisections have $(g,k,p,b)$ respectively equal to $(2,2,0,3)$,  $(3,2,1,1)$, $(5,3,0,3)$.}
\label{fig:triplanediagram}
\end{figure}

To obtain a relative trisection of $X^4\setminus \nu(S)$ with $p=0,b=3$ (which will be desired in Section \ref{sec:gettrisection}), we fix one intersection of $X_1\cap X_2\cap X_3\cap S$ and choose the boundary-stabilizations to never meet the corresponding boundary component of $X'_1\cap X'_2\cap X'_3$. This ensures that the resulting triple-intersection of the relative trisection on $X^4\setminus \nu(S)$ has three boundary components.

\subsection{Trisecting complements of arbitrary surfaces}\label{sec:general}
In this section, we trisect the complement of an arbitrary surface $S$ in $X^4$. The construction is similar to $\RP^2$ case. 
 Many indices are included for the very-interested reader. Averagely-interested readers may ignore these numbers.

Let $S\subset X^4$ be a connected surface with $\chi(S)=\chi$, in a $(g,k)$-trisected $4$-manifold $X^4=(X_1,X_2,X_3)$. Isotope $S$ so that $(X_1,X_2,X_3)$ induces a $(c,b)=((\chi+b)/3,b)$-trisection of $S$. By \cite{newsurfacetrisections}, we can stabilize each $X_i$ $((\chi+b)/3-1)$ times so that $(\chi+b)/3=c=1$ and $b=3-\chi$ (this increases $g$ and $k$). Delete a tubular neighborhood of $S$; let $X'_i:=X_i\setminus\nu(S)$. 

As in Subsection \ref{sec:rp2subsection}, take $C_1,C_2,C_3$ to be collections of $2-
\chi$ disjoint arcs with endpoints on $X'_1\cap  X'_2\cap  X'_3$, with $C_i\subset X'_j\cap X'_k\cap\boundary$. Take each arc in $C_i$ to be parallel to a distinct arc in $\nu(S)\cap X_j\cap X_k$; there is exactly one arc in $S\cap X_j\cap X_k$ which is not parallel to any arc in $C_i$. Moreover, take each arc of $C_i$ to twist zero times around the $3$-dimensional tubes of $\nu(S)\cap (X_j\cap  X_k)$ (i.e. we ask $S\cap(X_j\cap  X_k)$ cobounds disks $D$ in $X_j\cap  X_k$ with arcs in $X_i\cap  X_j\cap  X_k$ so that $C_i\cap D=\emptyset$). Finally, take $C_1,C_2,$ and $C_3$ to have disjoint endpoints.

Let $(\newtilde{X}_1,\newtilde{X}_2,\newtilde{X}_3)$ be the result of boundary-stabilizing $(X'_1, X'_2, X'_3)$ along every arc in $C_1\cup C_2\cup C_3$ (i.e. boundary-stabilizing $X'_1$ along $C_1$, $X'_2$ along $C_2$, and $X'_3$ along $C_3$).

\begin{proposition}
$(\newtilde{X}_1,\newtilde{X}_2,\newtilde{X}_3)$ is a relative trisection of $X^4\setminus\nu(S)$.
\end{proposition}

\begin{proof}
Recall $X'_i$ is a $4$-dimensional handlebody, and $\newtilde{X}_i$ is obtained from $X'_i$ by attaching $1$-handles. Therefore, $\newtilde{X}_i$ is a $4$-dimensional handlebody. In fact, $\newtilde{X}_i= X'_i\natural_{(2-\chi)}(S^1\times B^3)\cong\natural_{k+3-\chi} S^1\times B^3$. 

Moreover, $\newtilde{X}_i\cap\newtilde{X}_j$ is formed by attaching $3$-dimensional $1$-handles to the $3$-dimensional handlebody $X'_i\cap X'_j$, so $\newtilde{X}_i\cap\newtilde{X}_j$ is a $3$-dimensional handlebody. In fact,
$\newtilde{X}_i\cap\newtilde{X}_j=(X'_i\cap X'_j)\natural_{2(2-\chi)}(S^1\times D^2)\cong\natural_{g+7-3\chi} S^1\times D^2$.

Furthermore, the triple intersection
$\newtilde{X}_1\cap\newtilde{X}_2\cap\newtilde{X}_3$ is formed from $X'_1\cap X'_2\cap X'_3$ by attaching $3(2-\chi)$ orientation-preserving bands. Therefore, $\newtilde{X}_1\cap\newtilde{X}_2\cap\newtilde{X}_3$ 
is a connected, orientable surface of Euler characteristic $2-2g-2(3-\chi)-3(2-\chi)=-10-2g+5\chi$. The genus and number of boundary components of this triple intersection depends on the choice of $C_1$, $C_2$, and $C_3$. 
To finish the proof, we need to show that the proposed relative trisection induces an open book decomposition on the boundary.

\begin{claim}
 $\newtilde{X}_1\cap\boundary(X^4\setminus \nu(S))$ is a product over $\newtilde{X}_1\cap \newtilde{X}_2\cap\boundary(\newtilde{X}^4\setminus \nu(S))$ and over $\newtilde{X}_1\cap \newtilde{X}_3\cap\boundary (\newtilde{X}^4\setminus\nu(S))$. Moreover, these product structures agree.
\end{claim}

\begin{proof}
The proof is virtually the same as in Section \ref{sec:rp2subsection}; See Figure \ref{fig:pageinboundary}. 
Note $X'_1\cap X'_2\cap\boundary(X^4\setminus\nu(S))$ consists of $3-\chi$ longitudinal annuli on the solid torus $X'_1\cap\boundary(X^4\setminus\nu(S))$. After boundary-stabilizing $X'_3$ a total of $2-\chi$ times along $C_3$,
\begin{align*}
 X'_1\cap\boundary&\cong S^1\times B^2,\\
 X'_1\cap  X'_2\cap\boundary&\cong \sqcup_{2-\chi}D^2\sqcup (S^1\times I),\\
 X'_1\cap  X'_3\cap\boundary&\cong \sqcup_{3-\chi} (S^1\times I)\cup(\text{$2-\chi$ bands})\cong\Sigma_{0,4-\chi}.
\end{align*}

In the above, we implicitly use the fact that the intersection of $S$ with $X_1$ is a single disk. This ensures that attaching bands to $X'_1\cap  X'_3\cap\boundary$ parallel to arcs $C_3$ in $X'_1\cap X'_2\cap\boundary$ does not increase the genus of $X'_1\cap X'_3\cap\boundary$, as each successive band must join two distinct components.

After further boundary-stabilizing $X'_2$ $2-\chi$ times along $C_2$, we have
\begin{align*}
 X'_1\cap\boundary&\cong S^1\times B^2,\\
 X'_1\cap  X'_2\cap\boundary&\cong \sqcup_{2-\chi}D^2\sqcup (S^1\times I)\cup(\text{$2-\chi$ bands})\cong S^1\times I,\\
 X'_1\cap  X'_3\cap\boundary&\cong \Sigma_{0,4-\chi}\setminus\nu(\text{$2-\chi$ arcs})\cong S^1\times I.
\end{align*}

In the above, we implicitly use the fact that the intersection of $S$ with $X_1$ is a single disk. This ensures that attaching bands to $ X'_1\cap  X'_2\cap\boundary$ parallel to arcs $C_2$ in $ X'_1\cap  X'_3\cap\boundary$ does not increase the genus of $ X'_1\cap  X'_2\cap\boundary$, as each successive band must joint two distinct components. We also use the fact that the arcs $C_2$ meet $2(2-\chi)-2$ distinct boundary components of $ X'_1\cap X'_2\cap X'_3$. Each successive deletion from $ X'_1\cap  X'_3\cap\boundary$ must decrease the number of boundary components.

Thus, after performing the $ X'_2$ and $ X'_3$ boundary-stabilizations, $ X'_1\cap\boundary$ is a solid torus and each of $ X'_1\cap  X'_2\cap\boundary$ and $ X'_1\cap  X'_3\cap\boundary$ is an annulus on $\boundary ( X'_1\cap\boundary)$ whose core is a longitude of $ X'_1\cap\boundary$. At this point, we have a product structure $ X'_1\cap\boundary \cong (S^1\times I\times I)$, where $ X'_1\cap  X'_2\cap\boundary=S^1\times I\times 0$ and $ X'_1\cap  X'_3\cap\boundary=S^1\times I\times 1$.

The $ X'_1$ boundary-stabilizations along $C_1$ increases the genus of $ X'_1\cap\boundary$ by $2-\chi$, yielding $\newtilde{X}_1\cap\boundary \cong\natural_{3-\chi} S^1\times D^2$. On each of the $2-\chi$ solid tubes added to $ X'_1$ to form $\newtilde{X}_1$, there is a band on the boundary (parallel to the core of the tube) contained in $\newtilde{X}_1\cap \newtilde{X}_2\cap\boundary$ and another band on the boundary (parallel to the core of the tube) contained in $\newtilde{X}_1\cap \newtilde{X}_3\cap\boundary$. The effect on the product structure of $ X'_1\cap\boundary$ is to add $(2-\chi)$ product tubes of the form (band)$\times I$.

\end{proof}

The above claim holds similarly for $\newtilde{X}_2\cap\boundary$ and $\newtilde{X}_3\cap\boundary$, interchanging the roles of $\newtilde{X}_1, \newtilde{X}_2$ and $\newtilde{X}_3$. 
Thus, $(\newtilde{X}_1,\newtilde{X}_2,\newtilde{X}_3)$ is a relative trisection for $X^4\setminus \nu(S)$.
\end{proof}

Now we will explicitly trisect the complement of a specific surface $S$. We start from a shadow diagram $(\Sigma,\alpha,\beta,\gamma,s_\alpha,s_\beta,s_\gamma)$ for $S$. We stabilize $X_1$, $X_2$, and $X_3$ until each of $s_\alpha,s_\beta,s_\gamma$ consists of $3-\chi$ arcs, as in \cite{newsurfacetrisections}. Identify $\Sigma$ with $X_1\cap X_2\cap X_3$. Now choose $2-\chi$ distinct components of $s_\alpha,s_\beta, s_\gamma$ (each) to be (parallel to) $C_3$, $C_1$, and $C_2$ respectively. We obtain a relative trisection diagram $(\Sigma',\alpha',\beta',\gamma')$ for $(\newtilde{X}_1,\newtilde{X}_2,\newtilde{X}_3)$ by doing the following:

\begin{itemize}
\item Delete open neighborhoods of the endpoints of $s_*$ from $\Sigma$. Attach an orientation-preserving band for each component of $C_1, C_2$, and $C_3$, with endpoints of the band around endpoints of the arc component. Call the resulting surface $\Sigma'$.
\item Take $\alpha\subset \alpha',\beta\subset\beta',\gamma\subset\gamma'$.
\item For each arc in $C_3$, obtain an $\alpha'$ curve (shadow of $C_3)\cup($core of band corresponding to $C_3)$. Similarly obtain a $\beta'$ and $\gamma'$ curve for each component of $C_1$ and $C_2$, respectively.
\item For each arc in $C_3$, obtain an $\alpha'$ curve which encircles the shadow of $C_3$. Similarly obtain $\beta'$ and $\gamma'$ curves by encircling the shadows of $C_1$ and $C_2$, respectively.
\end{itemize}

This yields $g+4-2\chi(S)$ linearly independent $\alpha'$ curves on $\Sigma'$. As in Subsection \ref{sec:rp2subsection}, each $\alpha'$ curve bounds a disk in $\newtilde{X}_1\cap\newtilde{X}_2$. Moreover, we note $\chi(\Sigma')=2-2g-2(3\chi)-3(2-\chi)=-10-2g+5\chi$, so $\chi(\Sigma'\times I)\cong\partial_{11+2g-5\chi} (S^1\times B^3)$. Recall $\newtilde{X}_1\cap\newtilde{X}_2\cong\natural_{g+(3-\chi(S))+2(2-\chi(S))}(S^1\times B^2)\cong\natural_{g+7-3\chi(S)}(S^1\times B^2)$. Therefore, in a relative trisection diagram for $(\newtilde{X}_1,\newtilde{X}_2,\newtilde{X}_3)$ there are $g+4-2\chi$ distinct $\alpha'$ curves altogether. Thus, we have listed the complete set of $\alpha'$ curves (and similarly $\beta'$ and $\gamma'$ curves). 

In Figure \ref{fig:torusexample}, we consider $S=\text{(spun trefoil)}\#\text{(unknotted torus)}\subset S^4$ (a triplane diagram can be obtained by connect-summing diagrams from \cite{surfacetrisections}; we convert this into a shadow diagram). The torus $S$ can be isotoped so that the standard $(0,0)$-trisection $(X_1,X_2,X_3)$ of $S^4$ induces a $(2,6)$-bridge trisection of $S$. We stabilize each $X_i$ once along an arc in $X_j\cap X_k\cap S$ to find a $(3,1)$-trisection of $S^4$ inducing a $(1,3)$-bridge trisection of $S$; the shadow diagram $(\Sigma,\alpha,\beta, \gamma, s_\alpha, s_\beta, s_\gamma)$ of Figure \ref{fig:torusexample} illustrates this bridge trisection. To obtain a relative trisection of $S^4\setminus\nu(S)$, we choose two arcs in each $X_j\cap X_k\cap S$ be parallel to $C_i\subset\boundary\nu(S^4\setminus\nu(S))$. In Figure \ref{fig:torusexample}, we indicate shadows of $C_i$ in $s_*$. We then delete $\nu(S)$ and boundary stabilize each $X'_i:=X_i\setminus\nu(S)$ (twice) along $C_i$ to obtain a relative trisection $(\newtilde{X}_1,\newtilde{X}_2,\newtilde{X}_3)$ of $S^4\setminus\nu(S)$.

To obtain the relative trisection diagram $(\Sigma',\alpha',\beta',\gamma')$ of $(\newtilde{X}_1,\newtilde{X}_2,\newtilde{X}_3)$ pictured in Figure \ref{fig:torusexample} (bottom), we remove open neighborhoods of $\boundary s_*$ from $\Sigma$ and attach six bands with ends at the the boundary of $C_1, C_2, C_3$. Then we include two $\alpha'$ curves for each arc in $C_3$: one curve is the shadow of $C_3$ plus a core of the corresponding band, while one curve encircles the shadow of the arc in $C_3$. We similarly add four $\beta'$ and $\gamma'$ curves (each) corresponding to $C_1$ and $C_2$, respectively. In recap, there are seven total $\alpha'$ curves, given by:

\begin{itemize}
\item The $\alpha$ curves in the trisection $(X_1,X_2,X_3)$ of $S^4$. In this example, $(X_1,X_2,X_3)$ is a $(3,1)$-stabilization, so there are three such $\alpha'$ curves.
\item Curves $a\cup a'$, where $a$ is a shadow of an arc in $C_3$ and $a'$ is a core of the band corresponding to that component of $C_3$. There are two such $\alpha'$ curves in this example.
\item Curves encircling the shadows of $C_3$. In this example there are two such $\alpha'$ curves.
\end{itemize}

The seven $\beta',\gamma'$ curves are similarly related to $\beta, C_1$ and $\gamma, C_2$, respectively.

\begin{figure}
\includegraphics[width=\textwidth]{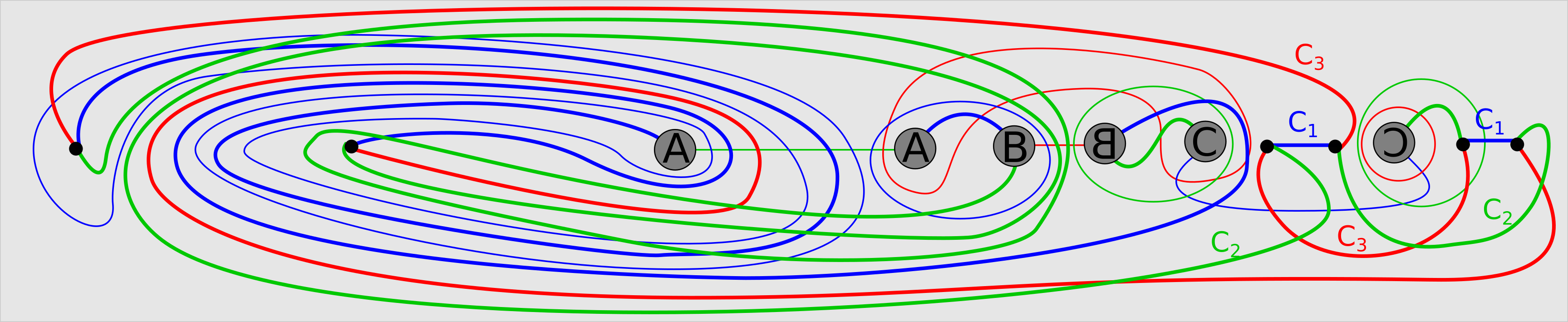}\\
\vspace{.1in}
\includegraphics[width=\textwidth]{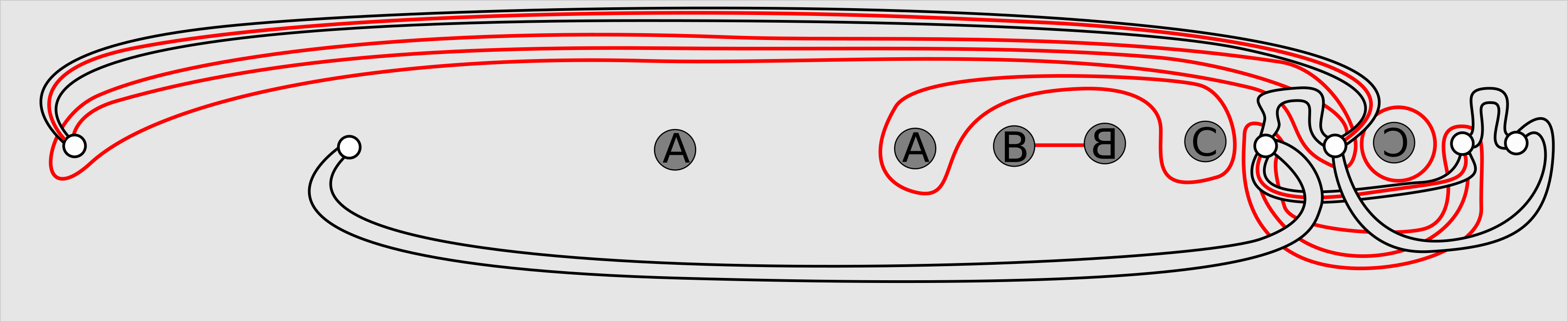}\\
\vspace{.1in}
\includegraphics[width=\textwidth]{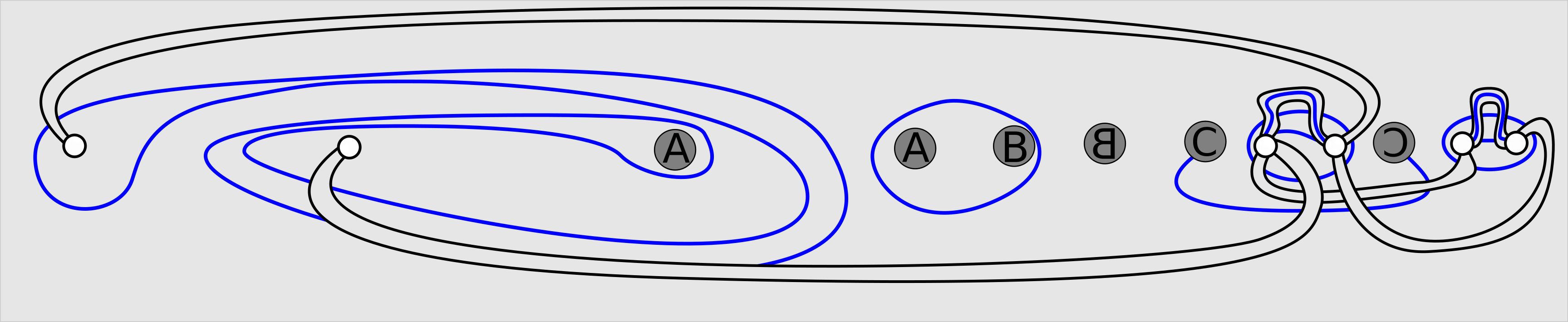}\\
\vspace{.1in}
\includegraphics[width=\textwidth]{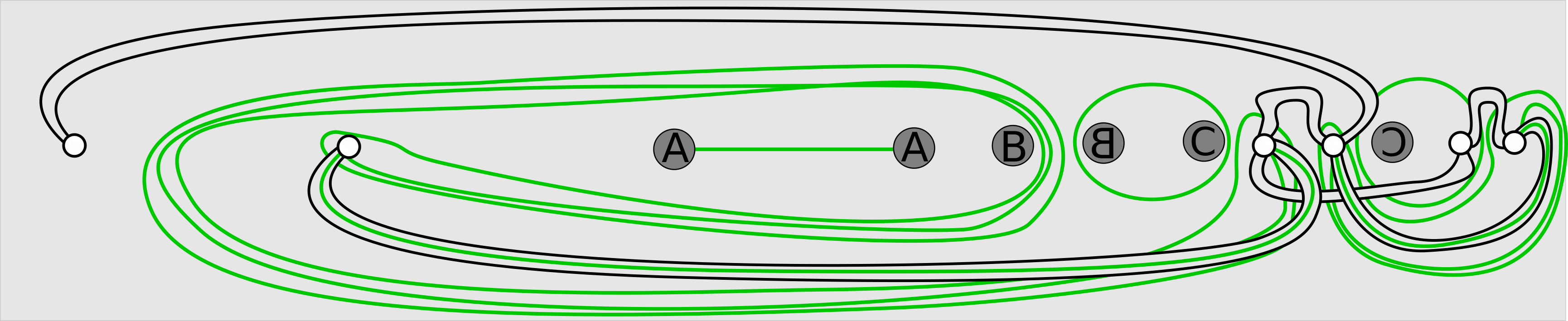}\\
\vspace{.1in}
\includegraphics[width=\textwidth]{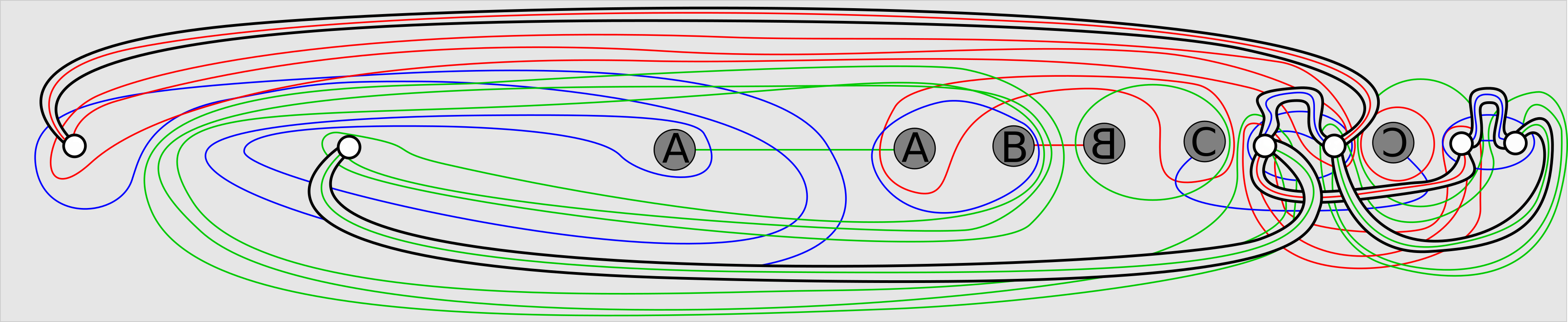}
\caption{Top: a shadow diagram for a $S=($spun trefoil$)\#($unknotted torus$)$ in $S^4$. Note $S$ is in $(1,3)$-bridge position. We indicate two arcs in each $s_*$ that comprise $C_1,C_2$, and $C_3$ in the construction of a relative trisection of $S^4\setminus\nu(S)$. We will obtain a diagram $(\Sigma',\alpha',\beta',\gamma')$ of this relative trisection. Second row: $\Sigma'$ and $\alpha'$. Third row: $\Sigma'$ and $\beta'$. Fourth row: $\Sigma'$ and $\gamma'$. Bottom: the relative trisection $(\Sigma',\alpha',\beta',\gamma')$. This relative trisection of $S^4\setminus\nu(S)$ has $(g,k,p,b)=(8,4,1,2)$.}
\label{fig:torusexample}
\end{figure}

\section{Gluing $P_{\pm}$ to $X^4\setminus\RP^2$.}\label{sec:gettrisection}
Let $S\subset X^4$ be an $\RP^2$ with Euler number $e(S)=\pm 2$. We have previously produced preferred $(g,k,p,b)=(2,2,0,3)$-trisections $T_1$, $T_2$ (or $\overline{T_1},\overline{T_2}$, if $e(S)=-2$) of $\nu(S)$ (see Section \ref{sec:preferredtrisections}), and can produce a $(g,k,0,3)$-trisection of $X^4\setminus\nu(S)$ via Section \ref{sec:rp2subsection}. The following easy lemma allows us to glue these trisections.

\begin{lemma}
Suppose $Q$ has an open book where the pages are $3$-punctured spheres. Then the monodromy of the open book consists of two Dehn twists around each boundary, not all of the same sign.
\end{lemma}
\begin{proof}
Suppose the monodromy of the open book consists of $a,b,c$ Dehn twists about the three boundary components correspondingly, for $a,b,c\in\Z$. Then $Q$ is a Seifert fibered space over $S^2$ with three exceptional fibers of orders $a$, $b$, and $c$. Then $\pi_1(Q)=\langle x_1,x_2,x_3,h\mid [x_i,h]=x_1^ah=x_2^bh=x_3^ch=x_1x_2x_3=1\rangle$. But recall also that $\pi_1(Q)$ is the quaternion group. Then $h\in Z(\pi_1(Q))$ implies $h=\pm 1$.

We have $\Z/2\oplus\Z/2=H_1(Q)$ is the abelianization of $\pi_1(Q)$. So if $h=1$, $\Z/2\oplus\Z/2\cong\langle x_1,x_2\mid ax_1=bx_2=c(x_1+x_2)=0\rangle$. This implies $a$ and $b$ are even, but neither $x_1$ nor $x_2$ can be $\pm 1$ (or else $H_1(Q)$ would be cyclic). Therefore, $4\mid(a,b,c)$, giving abelianization $\Z/4\oplus\Z/4$, a contradiction. Thus, $h=-1$.

Now $h=-1$, and $\Z/2\oplus\Z/2=\pi_1(Q)/\langle -1\rangle=\langle x_1,x_2,x_3\mid x_1^a=x_2^b=x_3^c=x_1x_2x_3=1\rangle$. By multiplying the $x_i$ by $-1$ and/or replacing $x_i$ with $x_i^{-1}$, we see this group is isomorphic to the triangle group $\langle x,y,z\mid x^{|a|}=y^{|b|}=z^{|c|}=xyz=1\rangle$. Since $\pi_1(Q)/\langle -1\rangle=\Z/2\oplus\Z/2$ is finite and dihedral, this is a spherical triangle group with $|a|=|b|=|c|=2$. Computing $4=|H_1(Q)|=|ab+bc+ca|$ yields $\{a,b,c\}=\pm\{2,2,-2\}$.
\end{proof}

\begin{corollary}\label{cor:rightsign}
Let $T$ be a $(g,k,0,3)$-trisection of $X^4\setminus \nu(S)$ produced by the algorithm of Section \ref{sec:rp2subsection}. If $e(S)=2$, then the monodromy on the open book induced by $T$ has left-handed twists about two bindings and right-handed twists about the other (mirrored for $e(S)=-2$).
\end{corollary}
\begin{proof}
Say $e(S)=2$. Fix $Q=\boundary P_+$ with singular fibers $S_0,S_1,S_{-1}$. We saw in Section \ref{sec:preferredtrisections} that $T_1,T_2$ each have monodromy consisting of right-handed twists about two boundaries and left-handed twists about the other. Suppose the same is true for $T$. If the left-handed boundary corresponds to $S_{-1}$, then we see $T_1$ and $T$ induce the same orientation on $Q$. Similarly, if the left-handed boundary corresponds to $S_0$ or $S_{1}$, then we see $T$ induces the same orientation on $Q$ as $T_2$. In either case, we find $X^4\setminus P_+$ and $P_+$ induce the same orientation on $Q$; a contradiction.
\end{proof}

\begin{corollary}\label{cor:lastcor}
Let $T$ be a $(g,k,0,3)$-trisection of $X^4\setminus \nu(S)$. Say $e(S)=2$. We may glue $T$ to $T_1$ or $T_2$ to obtain $(g+4,k)$-trisections of $X^4,\tau_S(X^4)$, and $\Sigma_S(X^4)$. (If $e(S)=-2$, then the result holds for gluing $T$ to $\overline{T_1}$ or $\overline{T_2}$.)
\end{corollary}

\begin{proof}
By Corollary \ref{cor:rightsign}, $\mirror{T},T_1,T_2$ induce homeomorphic open books on $Q$. By gluing $T$ and $T_1$, we may identify $S_{-1}\subset Q=\boundary P_+$ with the right-hand twist boundary of $T$. By gluing $T$ and $T_2$, we may identify $S_{-1}$ with either of the left-hand twist boundaries of $T$. Thus, we produce trisections of three $4$-manifolds $\nu(S)\cup_{\phi} (X^4\setminus \nu(S))$, where $\phi:Q\to Q$ preserves the Seifert fiber structure and can be chosen to map $S_{-1}$ to any singular fiber. By the discussion in Section \ref{sec:pricetwist}, these manifolds are $X^4,\tau_S(X^4),\Sigma_S(X^4)$.

To see that the resulting trisection is a $(g+4,k)$-trisection, recall that $T_i$ is a $(g',k',p,b)=(2,2,0,3)$-relative trisection. Then the result of gluing $T$ to $T_i$ is a $(g+g'+(b-1),k+k'-(2p+b-1))=(g+4,k)$-trisection.
\end{proof}

\subsection{Example}\label{exampletwist}
Finally, we present an example of trisecting the result of $\RP^2$ surgery.
Let $S\cong\RP^2\subset S^4$ be the connect-sum of the spun trefoil and an unknotted $\RP^2$ with Euler number $-2$.
We first isotope $S$ to be in $(c,b)=(1,2)$-bridge position with respect to a $(3,1)$-trisection of $S^4$. We depict a shadow diagram for $S$ in Figure \ref{fig:rp2example1} (top). This diagram can be obtained by understanding  \cite{surfacetrisections} (either the explicit example diagrams of twist-spun knots or the procedure to turn a movie of a knotted surface into a triplane diagram). We obtain a relative trisection $T=(\Sigma',\alpha',\beta',\gamma')$ of $S^4\setminus\nu(S)$ as in Section \ref{sec:rp2subsection}. A diagram of $T$ is pictured in Figure \ref{fig:rp2example1} (bottom), obtained as per the algorithm of Section \ref{sec:rp2subsection}. We specifically choose the arcs $C_1, C_2, C_3$ so that $\Sigma'$ has three boundary components.

To achieve surgery on $S$ diagramatically, we will glue $T$ to $\overline{T_1}$ or $\overline{T_2}$ (our two preferred trisections of $P_-$ from Section \ref{sec:prefer}) as in Corollary \ref{cor:lastcor}. By Price \cite{pricepaper} (see Section \ref{sec:price}), the gluing of $S^4\setminus\nu(S)$ and $P_-$ is determined up to diffeomorphism by the identifications of $\boundary \Sigma'$ with the boundary of the trisection surface of $\overline{T_i}$. In particular, the gluing is determined up to diffeomorphism by the choice of which boundary of $\Sigma'$ is identified with the $S_{-1}$ boundary component of $\overline{T_i}$. When gluing $T$ to $\overline{T_1}$, the resulting manifold is therefore determined up to diffeomorphism. There are potentially two nondiffeomorphic choices resulting from gluings of $T$ to $\overline{T_2}$.

To glue $T$ and $\overline{T_i}$, we must know the monodromy $T$ induces on $\boundary(S^4\setminus\nu(S))$. We apply the monodromy algorithm of \cite{monodromypaper}. One need not perform the entire algorithm -- the effect of the monodromy on one arc between the two leftmost boundary components of $\Sigma'$ is to add two right-handed twists around the leftmost boundary components and two right-handed twists about the middle boundary component. By Corollary \ref{cor:rightsign}, the twists about the third boundary are right-handed.

The boundary of $\Sigma'$ which does not meet any of the bands coming from boundary-stabilization (rightmost in Figure \ref{fig:rp2example2}) corresponds to a meridian of $S$. Gluing the $S_{-1}$ boundary of $\overline{T_2}$ to this boundary yields a trisection diagram of $S^4$.

Another (leftmost in Figure \ref{fig:rp2example2}) boundary of $\Sigma'$ corresponds to a curve in $\boundary(S^4\setminus \nu(S))$ which bounds a disk in $S^4\setminus\nu(S)$. Gluing the $S_{-1}$ singular fiber of $\boundary P_-$ to the corresponding singular fiber of $\boundary(S^4\setminus\nu(S))$ yields a manifold with nontrivial $H_1$; this is $\tau_S(S^4)$. Then gluing the $S_{-1}$ boundary of $\overline{T_1}$ to this boundary of $\Sigma'$ yields a trisection diagram of $\tau_S(S^4)$. 

Gluing the $S_{-1}$ boundary of $\overline{T_2}$ to the final boundary of $\Sigma'$ (middle in Figure \ref{fig:rp2example2}) yields a trisection diagram of $\Sigma_S(S^4)$.

We depict all the described gluings schematically in Figure \ref{fig:rp2example2}. This relative trisection diagram of $T$ agrees with the diagram of Figure \ref{fig:rp2example1} up to a surface automorphism and one handle slide of the $\gamma'$ curves.

\begin{figure}
\includegraphics[width=\textwidth]{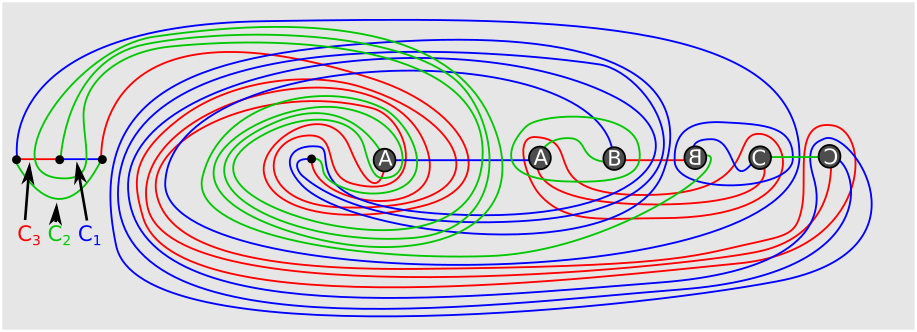}\\
\vspace{.1in}
\includegraphics[width=\textwidth]{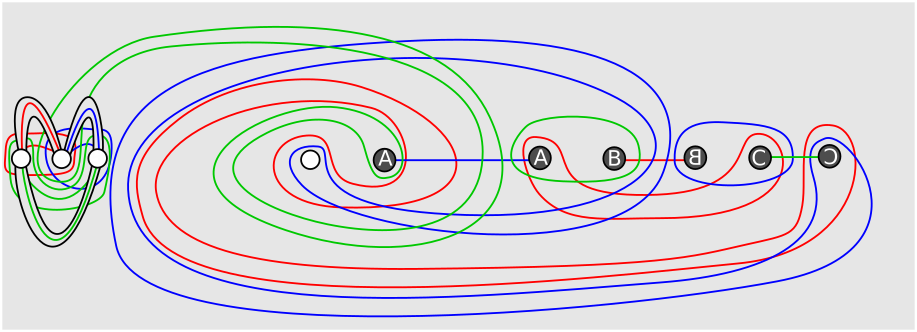}
\caption{Top: A shadow diagram for $S$. Here, $S$ is the connect sum of the spun trefoil with an unknotted $\RP^2$ (euler number $-2$) in $S^4$. We indicate shadows of arcs $C_1, C_2, C_3$ to be used in finding a relative trisection $T$ of $S^4\setminus\nu(S)$. Bottom: A relative trisection diagram for $T$. Here, $T$ has $(g,k,p,b)=(5,3,0,3)$.}
\label{fig:rp2example1}
\end{figure}

\begin{figure}
\includegraphics[width=\textwidth]{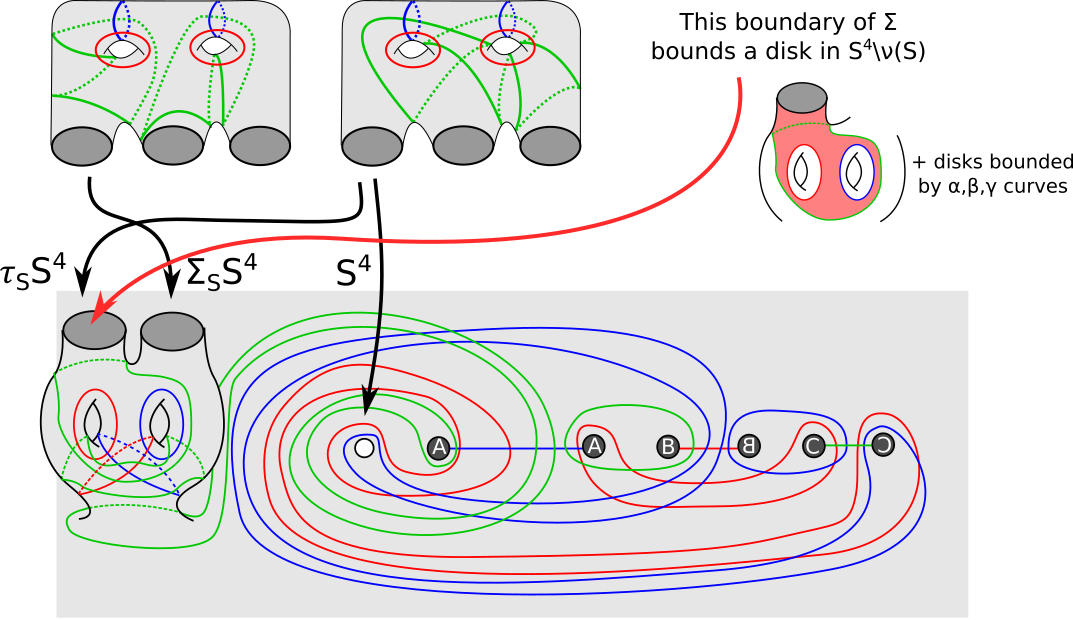}
\caption{We obtain trisections of $S^4, \tau(S^4),\Sigma(S^4)$ by gluing $\overline{T_1}$ or $\overline{T_2}$ to $T$. The choice of which boundary of $\Sigma'$ to identify with the $S_{-1}$ boundary of the trisection surface of $\overline{T_1}$ or $\overline{T_2}$ determines the diffeomorphism type of the resulting trisected $4$-manifold.}
\label{fig:rp2example2}
\end{figure}

\section{Further Questions}
In Example \ref{exampletwist}, we know $\Sigma_S(S^4)\cong S^4$ because $S=K\# P_-$, where $K$ is the spun trefoil. By \cite{katanaga}, $\Sigma_S(S^4)\cong \Sigma_K(S^4)\cong S^4$.

\begin{question}
Using Gay-Meier's trisections of Gluck twists \cite{jeffgluck} and our trisections of Price twists, is there a trisection-theoretic proof of  Katanaga, Saeki, Teragaito, Yamada's result \cite{katanaga}? That is, is there a trisection-theoretic proof that $\Sigma_K(X^4)\cong\Sigma_{K\# P_{\pm}}(X^4)$ for a $2$-knot $K$ with trivial Euler number? Possibly restricting to the case $X^4=S^4$?
\end{question}

When $\Sigma_S(S^4)\cong S^4$ we can ask about the complexity of the resulting trisection.

\begin{question}\label{finalquestion}
Let $S$ be a copy of $\RP^2$ embedded in $S^4$ so that $\Sigma_S(S^4)\cong S^4$. Let $T$ be a trisection of $\Sigma_S(S^4)\cong S^4$ arising from the algorithm of Section \ref{sec:gettrisection}. Is $T$ a stabilized copy of the standard $(0,0)$-trisection on $S^4$?
\end{question}

This question is a specific case of a question/conjecture of \cite{propr}.

\newtheorem*{jeffalexquestion}{Conjecture 3.11 of \cite{propr}}

\begin{jeffalexquestion}
Every trisection of $S^4$ is either the $(0,0)$-trisection or a stabilization of the $(0,0)$-trisection.
\end{jeffalexquestion}

In Figure \ref{fig:nonstandard}, we picture the simplest example of the trisections described in Question \ref{finalquestion}. This is a $(g,k)=(6,2)$-trisection of $S^4$ obtained by Price twisting $P_-\subset\RP^2$.
\begin{figure}
\includegraphics[width=\textwidth]{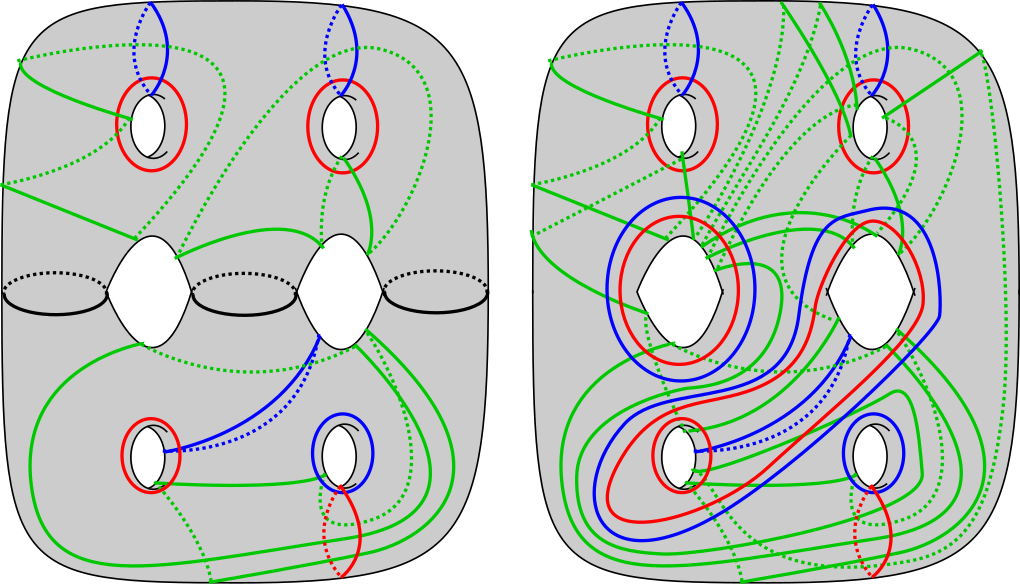}
\caption{Left: We begin to obtain a trisection $T$ of $S^4\cong\Sigma_{P_-}(S^4)$ by gluing $\overline{T_1}$ (top) to a relative trisection of $S^4\setminus\nu(P_-)$ (bottom), obtained as in Example \ref{exampletwist}. This is not yet a trisection diagram, as the surface is genus-$6$ but there are only $4$ (each) $\alpha,\beta$, and $\gamma$ curves. Right: We use the algorithm of \cite{monodromypaper} to find the two remaining $\alpha,\beta,\gamma$ curves (each) in $T$. Question \ref{finalquestion} asks: is $T$ a stabilization of the $(0,0)$-trisection of $S^4$?}
\label{fig:nonstandard}
\end{figure}

\bibliographystyle{abbrv}

\end{document}